\documentclass{article}
\usepackage{amsmath}
\usepackage{graphicx}
\usepackage{latexsym}
\usepackage{amssymb}
\usepackage{ifthen}
\usepackage[rgb]{xcolor}
\usepackage{tikz}
\usetikzlibrary{cd, arrows, matrix, intersections, math, calc}
\xdefinecolor{O}{RGB}{255, 102, 17}
\xdefinecolor{B}{RGB}{17, 87, 221}
\newtheorem{thm}{Theorem}[section]
\newtheorem{la}[thm]{Lemma}
\newtheorem{Defn}[thm]{Definition}
\newtheorem{Remark}[thm]{Remark}
\newtheorem{prop}[thm]{Proposition}
\newtheorem{cor}[thm]{Corollary}
\newtheorem{Example}[thm]{Example}
\newtheorem{Number}[thm]{\!\!}
\newenvironment{defn}{\begin{Defn}\rm}{\end{Defn}}

\newenvironment{rem}{\begin{Remark}\rm}{\end{Remark}}

\newenvironment{proof}{{\noindent\bf Proof.}}%
                  {\nopagebreak\hspace*{\fill}$\Box$\medskip\par}

\newcommand{\cD}{{\mathcal D}}

\newcommand{\cX}{{\mathfrak X}}
\newcommand{\cL}{{\mathcal L}}

\newcommand{\R}{{\mathbb R}}
\newcommand{\N}{{\mathbb N}}
\newcommand{\HH}{{\mathbb H}}

\newcommand{\Z}{{\mathbb Z}}
\newcommand{\E}{{\mathbb E}}
\newcommand{\T}{{\mathbb T}}
\newcommand{\eps}{\epsilon}

\newcommand{\Sph}{{\mathbb S}}

\newcommand{\g}{\gamma}

\newcommand{\wh}{\widehat}

\newcommand{\cg}{{\mathfrak g}}
\DeclareMathOperator{\id}{id}
\DeclareMathOperator{\ad}{ad}

\DeclareMathOperator{\vol}{vol}
\DeclareMathOperator{\pr}{pr}

\DeclareMathOperator{\Ham}{Ham}

\DeclareMathOperator{\Ad}{Ad}

\begin{document}
\begin{center}
{\bf\Large Stochastic Euler-Poincar\'e reduction for  central extension  }\\[4mm]
{\bf Ali Suri}
\end{center}
\begin{abstract}
\hspace*{-4.7mm} 
This paper explores the application of central extensions of Lie groups and Lie algebras to derive the viscous quasi-geostrophic (QGS) equations, with and without Rayleigh friction term, on the torus as  critical points of a stochastic Lagrangian. We begin by introducing central extensions and proving the integrability of the Roger Lie algebra cocycle $\omega_\alpha$, which is used to model the QGS on the torus. Incorporating stochastic perturbations, we formulate two specific  semi-martingales on the central extension and study the stochastic Euler-Poincaré reduction. Specifically, we add  stochastic perturbations to the $\mathfrak{g}$ part of the extended Lie algebra $\widehat{\mathfrak{g}} = \mathfrak{g} \rtimes_{\omega_\alpha} \mathbb{R}$ and prove that the resulting critical points of the stochastic right-invariant Lagrangian solve the viscous QGS equation, with and without Rayleigh friction term.
\vspace{2mm}
\end{abstract}
%
{\bf MSC. Primary 58J65, 22E65; Secondary 70H30, 35Q35, 60H10.}\\
{\bf Key words:}
Euler-Arnold equation, group of  diffeomorphisms, viscous quasi-geostrophic equations, stochastic Lagrangian, central extension, Roger cocycle.%
\tableofcontents
%
%
%
%
%
%
\newpage
\section{Introduction}
\textbf{Motivation.}\\[2.3mm]
Describing viscosity in the Navier-Stokes equations on compact manifolds via a stochastic variational approach has been studied by Nakagomi et al. \cite{Naka}, Cipriano and Cruzeiro \cite{Cip-Cru}, Arnaudon et al. \cite{A-C-C14}, and more recently for incompressible cases and advected quantities by Chen et al. \cite{CCR}.\\
In these studies, they considered a generalized derivative in their stochastic Lagrangian. The generalized derivative that measures the changes in the bounded variation part of a semi-martingale was first introduced by Nelson \cite{Nelson}  and has been used in various works  (e.g. \cite{A-C-C14,  CCR, Cip-Cru, Gliklikh, Koide-Kodama, Naka,  Yasue} and references therein). This derivative typically generates terms with second derivatives in the evolution equations, such as the viscosity term in the Navier-Stokes equation.

On the other hand, central extensions can be employed to characterize several evolution equations as geodesic equations (see, for example, Chapter 9 of \cite{M-M-O} and \cite{Viz1}). \\
The quasi-geostrophic equation on the torus was first studied from a geometric perspective by Zeitlin and Pasmanter \cite{Zeitlin}, and later using central extension by Vizman \cite{Viz2}. The central extension in the latter setting was responsible for generating the term that approximates the Coriolis force linearly according to its distance from the equator. In this model, a two-dimensional periodic fluid describes the atmosphere when considered as a thin layer or as a two-dimensional incompressible perfect fluid. Here, "perfect" means the fluid is non-viscous and homogeneous. These assumptions were necessary when Arnold, in 1966, introduced his approach  \cite{Arnold} to describe the Euler equations of a perfect fluid as a geodesic equation, now known as the Euler-Arnold equation, on the group of volume-preserving diffeomorphisms.\\
Motivated by \cite{A-C-C14}, we  show that a stochastic variational approach to incompressible viscous fluids on the torus $\T^2$ in the presence of the Coriolis effect can be described using the stochastic Euler-Poincaré reduction for central extensions. In this setting, we observe that the generalized derivative in our variational principle can generate both the viscosity term with second-order derivatives and a term without derivatives, referred to as "Rayleigh friction" or linear dampin in the system's evolution equations. The reason behind this is that transitioning from group 2-cocycles to Lie algebra 2-cocycles in the central extension setting requires two differentiations, which appear in the Euler-Arnold  equations. Nonetheless, it is crucial to note that the specific method of adding fluctuations (noise) to the particles, modeled by Brownian motion, generates terms according to the velocity field without any additional derivatives. This term will disappear if we consider the semi-martingale \eqref{hatg valued semimartingale 1,2}, which is  constructed using two vector fields on the torus.\\
In our setting, we utilize two Lie algebra-valued semi-martingales that are compatible with the configuration space in the sense that noise is added only to the fluid particles or, mathematically, to the Lie algebra of divergence-free vector fields, and not to the second part in $\wh{\cg} = \cg \rtimes_{\omega_\alpha} \mathbb{R}$. According to \cite{CCR}, the contraction matrix vanishes in this context. When considering the interaction of the first  semi-martingale \eqref{hatG v2 valued semimartingale torus} with the Lie algebra cocycle $\omega_\alpha$, it generates the Fourier series of the velocity field. In fact, by choosing an infinite dimensional orthogonal basis for the Hamiltonian vector fields generated by Kolmogorov flows, used in the definition of this process, we ensure that the rotation rate also remains constant. In the second approach, we consider the semi-martingale \eqref{hatg valued semimartingale 1,2} with two directions of fluctuation for the fluid particles. This stochastic differential equation leads to the evolution equations \eqref{Viscouse rot Euler} and \eqref{Viscous rot Euler stream} without the Rayleigh friction term, i.e., $\sigma=0$. \\
Fang also used the Gelfand-Fuchs 2-cocycles to prove the invariance of heat measures under the adjoint action of the circle and to build a canonical Brownian motion on the group of diffeomorphisms of the circle \cite{Fang}.\\
Moreover, Fang and Zhang \cite{F-Z} used spherical harmonics, which are counterparts of the Kolmogorov flows on the sphere, to deal with critical isotropic Brownian flows on the sphere.

Finally, it is worth mentioning that any changes in the setting such as the manifold $M$, the Lie algebra cocycle, or the way and direction of adding noise to the particles will alter the evolution equation.\\[2.3mm]
\textbf{Contributions and outline.}\\
In this paper we explore the application of central extensions of Lie groups and Lie algebras to derive the Euler-Arnold equations and its stochastic counterparts in geometric  hydrodynamics on torus.

In section \ref{Section pre}, we start by recalling the concepts of Lie group-valued semi-martingales and their generalized derivatives from \cite{A-C-C14}. Then, we proceed by discussing (deterministic) invariant Lagrangians on Lie groups and the Euler-Arnold equations.

In section \ref{Sec Cent ext}, we begin with the concept of central extensions of Lie groups and Lie algebras. Then, we prove the integrability of the Lie algebra cocycle called $\omega_\alpha$, which we will use to study the quasi-geostrophic equation (QGS) on the torus. The Roger cocycle $\omega_\alpha$ in this setting was first used by Vizman in \cite{Viz2}. However, the integrability of such cocycles on the torus in general is still an open problem. Using the notion of singular cocycles in light of the recent works of Janssens and Vizman \cite{Jan-Viz, Jan-Viz-2016}, we prove that the 1-dimensional Lie algebra central extension induced by $\omega_\alpha$ is integrable. This result will be used to show that the solutions of the stochastic processes in our approach lie in the proper space.\\
Then, we introduce an $L^2$ type  metric which represents the kinetic energy of the system on the 1-dimensional central extension  of Lie the algebra $\cg:=T_eG=T_e\cD_{\vol}^s(\mathbb{T}^2)$, $s>2$, and compute the co-adjoint operator. Afterwards, we derive the corresponding Euler-Arnold equations, covariant derivatives, and contraction operator within the central extension setting.

In section \ref{Sec stoch EPR}, incorporating stochastic perturbations, we formulate a stochastic process  on the central extension, which is a semi-martingale. Then,  using the  proper stochastic right-invariant Lagrangian, we study the stochastic Euler-Poincaré reduction along with the corresponding modified evolution equations.

In section \ref{Section Viscous QGS}, we first derive the viscous QGS 
\begin{equation}\label{Viscouse rot Euler}
\partial_tu  +  (u.\nabla)u  -  aTu   - {\nu}\Delta u + \sigma u    =  \mathrm{grad} p   \quad;\quad \mathrm{div} u=0 
\end{equation}
via a right-invariant stochastic Lagrangian.  More precisely, using standard Brownian motions, we add a stochastic perturbation to the $\cg$ part of the extended Lie algebra $\wh\cg=\cg\rtimes_{\omega_\alpha}\R$, as given by \eqref{hatG v2 valued semimartingale torus}.  Practically, this means that we consider the stochastic behavior of the fluid particles, not the parameter $(0,a)\in\wh\cg$ which is responsible for the approximation of the Coriolis force  (or rotation rate). It is shown that the corresponding process is a semi-martingale. We prove that the   semi-martingale given by \eqref{hatG v2 valued semimartingale torus} is a critical point of the stochastic 
right-invariant Lagrangian \eqref{stochastic energy functional} if and only if the drift part solves the  viscous QGS equation  \eqref{Viscouse rot Euler}, which serves as the evolution equation of the system. \\
If $\psi$ is the stream function of the velocity field $u$ then, we show that \eqref{Viscouse rot Euler} can be written as %
\begin{equation}\label{Viscous rot Euler stream}
\partial_t\Delta\psi+  \{\psi, \Delta\psi \}  +  \beta\partial_1\psi    - \nu\Delta^2 \psi  + \sigma \Delta\psi=0
\end{equation}
The term $ \beta\partial_1\psi$ is due to the Coriolis effect, $\nu \Delta^2 \psi$ represents turbulent viscosity, and $\sigma \Delta \psi$ is the Rayleigh friction. For $a = \nu = \sigma = 0$, we have the usual Euler-Arnold equation for a perfect incompressible fluid in its stream function form. 
The Rayleigh friction term is due to the specific way of adding stochastic perturbations to the fluid particles. 

In Section \ref{sec new example sigma=0}, we consider another semi-martingale on $\wh{g}$ using just two vector fields for fluctuation directions, as given by \eqref{hatg valued semimartingale 1,2}. We observe that this simpler form of randomness does not generate the Rayleigh friction term in \eqref{Viscouse rot Euler} and \eqref{Viscous rot Euler stream}. More precisely the semi-martingale generated by \eqref{hatg valued semimartingale 1,2} is a critical point of the stochastic Lagrangian \eqref{stochastic energy functional} if and only if the drift part of \eqref{hatg valued semimartingale 1,2} solves \eqref{Viscouse rot Euler} with $\sigma=0$. Moreover, if $u=\nabla^\perp\psi$, then it would satisfy \eqref{Viscous rot Euler stream} with $\sigma=0$.


%
%
%

\section {Preliminaries}\label{Section pre}
In this section, we review the concepts of Lie group-valued semi-martingales and their generalized derivative, as discussed in \cite{A-C-C14} and \cite{Emery}. We also explore the regularity of Lagrangians on Lie groups and derive the Euler-Arnold equations based on the framework presented in \cite{Jer-Cher, Ebin-Marsden}.
\subsection{Semi-Martingales on Lie groups}
A semi-martingale is a stochastic process that generalizes both martingales and finite variation processes. 
A semi-martingale taking values in a Lie group $G$ extends the concept of real-valued semi-martingales to the realm of Lie groups, incorporating the group's structure. Let $(S, \mathcal{F}, \mathbb{P})$ be a probability space with a filtration $(\mathcal{F}_t)_{t \geq 0}$. A $G$-valued semi-martingale is a process $X: [0, \infty) \times S \to G$
such that for any differentiable  function $f: G \to \mathbb{R}$, the process $f(X_t)$ is a real-valued semi-martingale.\\
Let $\nabla$ be a linear connection on $G$. A semi-martingale \(g\) with values in \(G\) is called a \(\nabla\)-local martingale if for every \(f \in C^{2}(G)\) the process 
\[
t \mapsto f(g (t)) - f(g (0)) - \frac{1}{2} \int_{0}^{t} \operatorname{Hess} f(g(s)) (d g(s),  dg(s)  )  \, ds
\]
is a real-valued local martingale (see e.g. \cite{Hsu} proposition 2.5.2). \\
Motivated by the Doob-Meyer decomposition, we define the $\nabla$-generalized derivative as follows: if for each differentiable map $f \in C^2(G)$, there exists $A(t) \in T_{g(t)}G$ almost surely (a.s.) such that
\[
N^f_t:=f(g(t))-f(g(0)) -\frac{1}{2}\int_0^t \operatorname{Hess} f(g(s)) (d g(s),  dg(s)  )  \, ds  - \int_0^tA(s)f(g(s))ds
\]
is a real-valued local martingale, then we define the $\nabla$-generalized derivative 
\[
\frac{D^{\nabla} g(t)}{d t} := A(t).
\]
Equivalently, the \(\nabla\)-generalized derivative of \(g(t)\) is defined as:
\[
\frac{D^{\nabla} g(t)}{d t} := \mathbf{P}_{t} \left( \lim_{\epsilon \rightarrow 0} \mathbb{E}_{t} \left[ \frac{\eta (t+\epsilon) - \eta(t)}{\epsilon} \right] \right),
\]
where \(\mathbf{P}_{t}: T_{e} G \rightarrow T_{g(t)} G\) represents the stochastic parallel translation, 
\[
\eta(t)=\int_0^t\mathbf{P}_s^{-1}\circ dg(s)\in\cg:=T_eG
\]
and \(\mathbb{E}_{t}[\cdot] = \mathbb{E}[\cdot \mid \mathcal{F}_{t}]\) denotes the conditional expectation  \cite{Emery}.
In the Euclidean space $\R^n$, the $\nabla$-generalized derivative is the usual generalized derivative (\cite{A-C-C14, CCR, Cip-Cru, Nelson, Yasue}).

One crucial aspect of Lie group-valued semi-martingales is their interaction with the group's algebraic structure, particularly through the group translation. \\
More precisely, we will   consider the $\cg$-valued semi-martingales of the form
\begin{eqnarray}\label{g valued semimartingale}
dg(t)  &=&   T_eR_{g(t)}  \Big(  \sum_{i=1}^m H_i\circ dW^i_t -\frac{1}{2}\nabla_{H_i}H_idt +u(t)dt   \Big)\nonumber \\
&  =&   T_eR_{g(t)}  \Big(  \sum_{i=1}^mH_i dW^i_t  +u(t)dt)  \Big)
\end{eqnarray}
where $\circ dW^i_t$  is the Stratonovich integral and   $dW^i_t$ represents  It\^o integral.   We recall that, $W_t$ is an $\R^m$-valued Brownian motion, $R$ denotes the right translation and $H_1,\dots,H_m, u\in\cg $. In the special case for the process \eqref{g valued semimartingale} we have
 \[
 \frac{D^{\nabla} g(t)}{d t} =T_eR_{g(t)}  u(t)
 \]\\
Practically, the martingale part is vanished by the conditional expectation and the $\nabla$- generalized derivative is given by the bounded variation part (or the drift) $u$. 
It has been shown that the $\nabla$-generalized derivative can produce the second-order differential terms associated with viscosity (\cite{A-C-C14, CCR, Cip-Cru}), which essentially originate from the Itô correction term. However, we will see that in the central extension setting, the lower-order terms also appear in the evolution equation.

%
%
%
%
\subsection{Euler-Arnold equations}\label{section Euler-Arnold eq}
Let $G$ be a smooth manifold modeled on the Banach space $\E$. Following \cite{Jer-Cher}, section 5, we say that a Lagrangian $L:TG
\longrightarrow \R$ is regular if, locally on any chart $U\subset G$, the induced map $D_2D_2L(x,v):\E\times \E\longrightarrow \R$ for any $(x,v)\in U\times \E\subseteq TG$ is weakly non-degenerate. Here, $D_2L$ denotes the derivative with respect to the second component of the local representative $L
:U\times \E\simeq TG|_U\longrightarrow \R$.

If $G$ also admits a topological group structure and the right (or left) translation is smooth, then we can consider right-invariant Lagrangians. More precisely, the Lagrangian is called right-invariant if $L(TR_h(g,\dot{g})) = L(g,\dot{g})$ (respectively left-invariant if $L(TL_h(g,\dot{g})) = L(g,\dot{g}) $), where $R_h:G 
\longrightarrow G$  is the right  ($L_h:G\longrightarrow G$ is the left) translation and $(g,\dot{g})$ is a typical tangent vector in $T_gG$. Usually, we will use the notation $\dot{g} h$ rather than $TR_h(g,\dot{g} )$. In this case we can consider the reduced Lagrangian defined by $l:\cg\longrightarrow \R$ defined by
\begin{equation*}
l(\dot{g} g^{-1}):=L(g,\dot{g})    
\end{equation*}
where $\cg:=T_e G$ is the Lie algebra of the group $G$. One can induce an invariant Lagrangian by considering a  bilinear symmetric positive definite map like  $\ll,\gg:\cg\times\cg\longrightarrow\R$.
Then, the \textbf{Euler-Lagrange} equations for the  right invariant Lagrangian $l(u,u)=\ll u,u \gg$, $u\in\cg$ is given by
\begin{equation}\label{Euler-Poincare equation}
\partial_t v=-\ad^*_vv
\end{equation}
where $v:(\eps,\eps)\longrightarrow \cg$ is a differentiable curve and the $\ad^*$ is given by 
\[
\ll \ad^*_uv,w\gg=\ll v,\ad_uw\gg =\ll v,-[u,w]\gg.
\]
Note that, when we deal with left invariant metrics, then $\ad_uv=[u,v]$.

In geometric hydrodynamics, $G$ is typically the group of volume-preserving diffeomorphisms on a compact Riemannian manifold $M$
\[
\cD^s_{\vol}(M)=\{\eta:M\longrightarrow M;~\eta^*\mu=\mu \textrm{ ~and $\eta$ is an $H^s$ diffeomorphisms}\}.
\]
The right translation is given by
\[
R_g:\cD^s_{\vol}(M)\longrightarrow \cD^s_{\vol}(M);\quad \eta\longmapsto \eta\circ g
\]
which is smooth. We remind that here, $H^s$ means that $\eta$ and its partial derivatives up to order $s$  are $L^2$. In the case that $s>\frac{\dim M}{2}+1$ then $\cD^s_{\vol}(M)$ admits a smooth manifold structure modelled on space of $H^s$ divergence free vector fields \cite{Ebin-Marsden}.
In the case that $G$ is the group of diffeomorphisms, we refer to equation \eqref{Euler-Poincare equation} as the \textbf{Euler-Arnold} equation.
%
%
%
%
%
\section{Central extension of Lie groups and Lie algebras }\label{Sec Cent ext}
The materials in this section about central extensions of Lie groups and Lie algebras are from \cite{M-M-O}, \cite{Jan-Viz}, \cite{KhesinWendt} and \cite{Ebin-Preston}. However, the section \ref{section Roger Cocycle int} concerning the integrability of the Roger cocycle is original.
\subsection{Central extension of Lie groups} Let $G$ be a Lie group and $A$ an Abelian Lie group. Consider the short exact sequence of Lie groups with smooth Lie group homomorphisms
\begin{equation}\label{Short exact seq Lie group}
e \longrightarrow A \stackrel{i}{\longrightarrow} \wh{G} \stackrel{p}{\longrightarrow} G \longrightarrow e 
\end{equation}
where $e$ denotes the trivial group. Then, $\wh{G}$ is an extension of $G$ by $A$ if $p$ admits a local smooth section $\sigma: U \longrightarrow \wh{G}$ such that $p \circ \sigma = \mathrm{id}_U$, where $U$ is an open neighborhood of the identity in $G$. $\wh{G}$ is called a central extension of $G$ by $A$ when $i(A)$ is in the center of $\wh{G}$. In this paper we will deal with central extensions.\\
If the sequence \eqref{Short exact seq Lie group} admits a global section $\sigma: G \longrightarrow \wh{G}$, or equivalently, if the sequence splits, then the extension can be constructed explicitly using Lie group 2-cocycles.
In the case that the extension is not split,  the concepts of pre-quantum bundles and the quantomorphism group are used. In this paper  $A=\Sph^1$.
%
%
%
%
\subsection{Central extension of  Lie algebras}
A central extension of a Lie algebra $\cg$ by an abelian Lie algebra $\mathfrak{a}$ is a Lie algebra $\wh\cg$ and a {split  short exact sequence of Lie algebra homomorphisms
\[
0 \longrightarrow \mathfrak{a} \stackrel{i}{\longrightarrow} \wh\cg \stackrel{p}{\longrightarrow} \cg \longrightarrow 0
\]
where the image of $i(\mathfrak{a})$ lies in the center of $\wh{\mathfrak{g}}$, meaning that for all $a \in \mathfrak{a}$ and $\tilde{X} \in \tilde{\mathfrak{g}}$, the elements commute: $[i(a), \tilde{X}] = 0$.
In this setup, $\wh\cg$ is called a central extension of $\cg$ by $\mathfrak{a}$, and the sequence expresses that $\wh\cg$ contains $\mathfrak{a}$ as a central subalgebra.

A 2-cocycle for a central extension of an algebra $\cg$ by an abelian algebra $\mathfrak{a}$ is a bilinear map $\omega: \cg \times \cg \rightarrow \mathfrak{a}$ satisfying the following properties. The cocycle condition holds true i.e. for all $u,v,w \in \cg$ we have
\[ \omega([u,v],w) + \omega([v,w],u) + \omega([w,u],v) = 0 \]
and the 2-cocycle is anti-symmetric, meaning $\omega(u,v) = -\omega(v,u)$.
In this case we endow $\wh\cg=\cg\times\mathfrak{a}$ with the Lie bracket
\begin{eqnarray}
    [(u,a),(v,b)]=([u,v]~,~\omega(u,v) )\quad;\quad  (u,a),(v,b)\in\wh\cg
\end{eqnarray}
which implies that $\wh\cg $ is a Lie algebra and $i(\mathfrak{a})$ lies in the center of $\wh\cg$. To emphasis that the Lie bracket is defined by $\omega $, usually we use the notation $\wh\cg:=\cg\rtimes_{\omega}\mathfrak{a}$. Moreover, the following short exact sequence of  Lie algebras 
\begin{equation*}
    0\longrightarrow \mathfrak{a} \stackrel{i}{\longrightarrow} \cg\rtimes_{\omega}\mathfrak{a} \stackrel{pr_1}{\longrightarrow}  \cg \longrightarrow 0 
\end{equation*}
splits where $i(a)=(0,a)$ and $pr_1(u,a)=u$. In this paper $\mathfrak{a}=\R$. 

In the sequel, we state the concept of the Lie algebra cohomology from \cite{Jan-Viz-2016}. For $n\in\N\cup\{0\}$ set $C^n(\cg,\R)$ to be the space of all $n$-linear alternating continuous maps $\alpha:\cg^n\longrightarrow \R$. Consider the differential $d:C^n(\cg,\R)  \longrightarrow    C^{n+1}(\cg,\R)$
which maps $\alpha$ to $d\alpha$ and 
\begin{eqnarray*}
d\alpha(x_0,\dots ,x_n):=\sum_{0\leq i< j\leq n} (-1)^{i+j} \alpha ([x_i,x_j],x_0,\dots, \wh{x}_i,\dots, \wh{x}_j,\dots, x_n)
\end{eqnarray*}
where $\wh{x}_i$ means that the entry in position $i$ is omitted. Moreover we define $d$ to be identically zero on $C^0(\cg,\R)$. We set $Z^n(\cg,\R):=\{\alpha\in C^n(\cg,\R)~;~d\alpha=0\}$ and  $B^n(\cg,\R):=\{\alpha\in C^n(\cg,\R)~;\exists~\gamma\in C^{n-1}(\cg,\R) \textrm{ s.t. } \alpha=d\gamma\}$.
The resulting cohomology is known as the Lie algebra cohomology and is denoted by $H^n(\cg,\R)$   where
\[
H^n(\cg,\R):=\frac{Z^n(\cg,\R)}{B^n(\cg,\R)}.
\]
\begin{rem}
For a compact and connected symplectic manifold $M$, set $\cD_{\Ham}(M)$ to be the group of all Hamiltonian diffeomorphisms on $M$. The corresponding Lie algebra is denoted by $\cX_{\Ham}(M):=T_e\cD_{\Ham}(M)$.  It is known that $H^2(\cX_{\Ham}(M),\R)\simeq H^1_{dR}(M)$ (see, e.g., \cite{Jan-Viz-2016} theorem 4.17).\\
If $M$ is prequantizable with a prequantum $\Sph^1$-bundle $\pi:P\longrightarrow M$ then $\cD_q(P)_0$ is a $\Sph^1$ central extension of $\cD_{\Ham}(M)$. We recall that $\cD_q(P)$ is the space of all diffeomorphisms which preserve the contact form $\theta$ on $P$  and $\cD_q(P)_0$ is the connected component of identity (for more details see e.g. \cite{Jan-Viz} section 4.1).  
\end{rem}

%
%
%
%
\subsection{Integrablility of a Roger cocycle on $\cD_{\vol}(\T^2)$   }\label{section Roger Cocycle int}
Construction of the central extension for the group of Hamiltonian diffeomorphisms on a surface with genus 1 (like a torus) from a Lie algebra cocycle is an open problem. In this paper, we will use a specific Roger Lie algebra 2-cocycle
\begin{equation}\label{Roger cocycle v1}
\omega_\alpha:\cg\times\cg\longrightarrow \R\quad ; (X_f,X_g)\longmapsto \int_{\T^2}f\alpha(X_g)d\vol
\end{equation}
and here we prove its integrability.  
In \eqref{Roger cocycle v1},  $\cg=T_e\cD_{\Ham}(\T^2)$ and in our case the closed 1-form $\alpha$ is $\alpha=-\frac{1}{2\pi} \theta_2$.

Consider the 1-form $\lambda=-ydx+xdy $ on $\R^2$ and restrict it to the unit circle $\Sph^1$. For the paramaetrization 
\[
\varphi:(0,2\pi)\longrightarrow \Sph^1\quad ;\quad t\longmapsto (\cos t,\sin t)
\]
we have $\varphi^*\lambda =dt$. Usually, this 1-form on $\Sph^1$ is denoted by $\theta$. Note that  $\int_{\Sph^1}\theta=2\pi$ which means that   this form is not exact.

For $\T^2=\Sph^1\times \Sph^1$ we consider the parametrization 
\[
\Phi:(0,2\pi)\times(0,2\pi) \longrightarrow \T^2\quad ;\quad (t,s) \longmapsto \Big( (\cos t,\sin t) , (\cos s,\sin s)  \Big)
\]
and the 1-forms $\theta_1, \theta_2\in\Omega^1(\T^2)$ which are given by $\Phi^*\theta_1=dt$ and $\Phi^*\theta_2=ds$ respectively.\\
A closed 1-dimensional submanifold $N\subset \T^2$ induces the   singular Lie algebra 2-cocycle 
\[
\omega_N:\cg\times \cg\longrightarrow \R\quad;\quad (X_f,X_g) \longmapsto \int_N fdg
\] 
and the closed 1-form $\alpha\in\Omega^1(\T^2)$ defines the Roger 2-cocycle
\[
\omega_{\alpha}: \cg\times \cg\longrightarrow \R  \quad;\quad (X_f,X_g) \longmapsto \int_{\T^2} f\alpha(X_g)d\vol_{\T^2}.
\]
%
%
%

%
\begin{prop}\label{prop Roger cocycle is int}
For  $\alpha=-\frac{1}{2\pi}\theta_2$ there exists a closed submanifold $N\subset \T^2$ such that  the lie algebra 2-cocycles $\omega_\alpha$ and $\omega_N$ are cohomologous.
\end{prop}
The proof is moved to the section Appendix.
%
%
%
\begin{cor}
For $\alpha=\beta \theta_2$, for suitable $\beta$ ,the Roger cocycle $\omega_\alpha$ integrates to a central $\Sph^1 $-extension    of the group of $\cD_q(P)_e$ of quantomorphisms. In the other words, the following short exact sequence 
\begin{equation}\label{SES quant}
    e\longrightarrow \Sph^1 \stackrel{i}{\longrightarrow}  \cD_q(P)_0\stackrel{\pi}{\longrightarrow} \cD_{\Ham}(\T^2)  \longrightarrow e 
\end{equation}
admits a local section $\rho:U\subseteq \cD_{\Ham}(\T^2)\longrightarrow \cD_q(P)_0$ with the property $\rho\circ\pi=\id_U$. 
The Lie algebra counterpart of the above sequence is the short exact sequence
\begin{equation*}
    0\longrightarrow \R \stackrel{i}{\longrightarrow} \cg\rtimes_{\omega_\alpha}\R \stackrel{pr_1}{\longrightarrow}  \cg \longrightarrow 0 
\end{equation*}
with $\cg=\cX_{\Ham}(\T^2).$ We recall that the Lie bracket on $\wh\cg=\cg\rtimes_{\omega_\alpha}\R$ is given by 
\[
[(x_f,a) , (X_g,b)]_{\wh\cg} = \big([x_f , X_g]_{\cg} ~ , ~ \omega_\alpha(X_f,X_g) \big).
\]
\end{cor}
The exact sequence \eqref{SES quant} implies that $\pi:\cD_q(P)_0\longrightarrow \cD_{\Ham}(\T^2)$ is an $\Sph^1$ principal
bundle and specially in a neighborhood  $U\subseteq\cD_{\Ham}(\T^2)$ of identity it looks like 
\[
 \cD_q(P)_0 \supseteq U\times\Sph^1  \stackrel{\pr_1}{\longrightarrow} U\subseteq\cD_{\Ham}(\T^2).
\]
As a consequence, if ${\g}:[0,\tau]\longrightarrow \cD_q(P)_0 $ is a differentiable (or continuous) curve which $\g(0)$ is identity then, there exists $\tau_0\leq \tau$ such that 
\begin{equation*}
\g(t):[0,\tau_0]\longrightarrow U\times \Sph^1\subseteq \cD_q(P)_0 \quad;\quad t\longmapsto (g(t),a(t))
\end{equation*}
where g takes its values in $U\subseteq \cD_{\Ham}(\T^2)$ and $a(t)$ is $\Sph^1$ valued. The group operation on $U\times \Sph^1$ is given by 
\begin{equation*}
(g_1,a_1).(g_2,    a_2):=(g_1g_2,  a_1a_2 \Omega(g_1,g_2))
\end{equation*}
where $\Omega$ is the group cocycle and 
\[
\Omega:U\times U\longrightarrow \Sph^1\quad;\quad (g_1,g_2)\longmapsto \rho(g_1)\rho(g_2)\rho(g_1g_2)^{-1}.
\]
In this case, for the right translation 
\[
\wh{R}_{(g,a)}:U\times\Sph^1\longrightarrow U\times\Sph^1
\]
the differential of the right translation is 
\begin{eqnarray}\label{Righ translation}
T_{(e,1)}\wh{R}_{(g_1,a_1)}(u,b) &=& \frac{d}{dt}|_{t=0} \wh{R}_{(g_1,a_1)}(g(t),a(t)) \\
&=& \frac{d}{dt}|_{t=0} \big(g_1g(t),a_1a(t) \Omega(g(t) , g_1) \big)\nonumber\\
&=& \big(T_eR_{g_1}u , \frac{d}{dt}|_{t=0} a_1a(t) \Omega(g(t) , g_1) \big)\nonumber
\end{eqnarray}
where $(g(t),a(t))$ is a curve in $U\times\Sph^1$ with the velocity 
\[
\frac{d}{dt}|_{t=0}(g(t),a(t))=(u,a)  \in \cX_{\Ham}(\T^2)\times\R.
\]
Sometimes we write $e^{i(a+b+ \Omega(g_1 , g_2) )}$ rather than $ab \Omega(g_1 , g_2)$ to emphasize that it lies in the circle $\Sph^1.$\\
The Lie algebra cocycle $\omega_\alpha$ is reproduced via the following formula
\begin{eqnarray}\label{Induced Lie algebra cocycle}
\omega_\alpha(u,v)=\frac{\partial^2}{\partial t\partial s}|_{t=s=0} \Omega(\eta_t,\xi_s)  -   \frac{\partial^2}{\partial t\partial s}|_{t=s=0} \Omega(\xi_s  ,\eta_t)
\end{eqnarray}
Here $\eta_t$ and $\xi_s$ are two curves in $G$  which generate the Lie algebra elements $u,v\in\cg$ i.e. $\partial_t|_{t=0}\eta_t=u$ and $\partial_s|_{s=0}\xi_s=v$ (see e.g. \cite{KhesinWendt}, page 26).
%
%
%
%
\begin{rem}\label{remark kirillov}
According to \cite{Kirillov} and \cite{Viz2}, the cocycle   \eqref{Roger cocycle v1} could be extended to a cocycle on the space of symplectic (or divergence free) vector fields by the conditions 
\[
\omega_\alpha(\partial_1,X_f) = \omega_\alpha(X_f,\partial_2)= \omega_\alpha(\partial_1,\partial_2)=0.
\]
According to proposition 6.1 \cite{Jan-Viz}, integrability of $\omega_{\alpha}$ on $\cX_{\Ham} (\T^2)$ implies that $\omega_\alpha$ is also integrable on the space of divergence free vector fields $T_e\cD_{\vol}^s(\T^2)$. 
\end{rem}
\subsection{Euler-Arnold equations on 1-dimensional central extension}
In this part, we compute the  Euler-Arnold equations on the 1-dimensional central extension $\wh{\cg}=\cg\rtimes_{\omega}\R$, where $\omega$ is any Lie algebra 2-cocycle and $\cg$ is any Lie algebra. In this respect, first we compute  the co-adjoint operator $\wh{\ad}^*:\wh{\cg}\longrightarrow \wh\cg$ according  to the adjoint operator of $\cg$ and the Lie algebra cocycle $\omega$. For    $(X,a),(Y,b)\in\wh{\cg}$ we consider the metric
\begin{equation}\label{L2 inner product on central extension}
\ll(X,a),(Y,b)\gg_{\wh\cg}:=\ll X,Y\gg_{\cg} +ab.
\end{equation}
Following the notations of \cite{Suri2023}, for  any $(X,a),(Y,b),(Z,c)\in\wh{\cg}$ we have
\begin{eqnarray*}
\ll  \wh{\ad}_{(X,a)}^* (Y,b)  ,  (Z,c) \gg_{\wh{\cg}}    &=&    \ll   (Y,b)  , \wh{\ad}_{(X,a)} (Z,c) \gg_{\wh{\cg}}   \\
&=& -   \ll   (Y,b)  , [(X,a), (Z,c)]_{\hat{\cg}} \gg_{\wh{\cg}}\\
&=&   - \ll   (Y,b)  , ([X,Z],  \omega(X,Z)) \gg_{\wh{\cg}}\\
&=&       \ll Y, -[X,Z]\gg_{\cg}   -b\omega(X,Z) \\
&=&       \ll Y, \ad_XY\gg_{\cg}   -b\omega(X,Z) \\
&=&       \ll \ad^*_XY  ,   Z\gg_{\cg}   -\ll b T X  ,  Z\gg_{\cg}   \\
&=&  \ll \ad^*_XY  - b TX   ,   Z\gg_{\cg}\\
&=&  \ll ( \ad^*_XY  - b TX,0)   ,   (Z,c)  \gg_{\wh{\cg}}\\
\end{eqnarray*}
As a result $\wh{\ad}_{(X,a)}^*:\wh{\cg}\longrightarrow \wh{\cg}$ maps $(Y,b)\in \hat{\cg}$ to $(\ad^*_XY-bTX , 0)$ where the operator $T:\cg\longrightarrow\cg$ is given by 
\begin{equation}\label{Operator T}
    \ll TX,Y\gg_\cg=\omega(X,Y)
\end{equation}
Moreover, for the curve $(u,a):(-\epsilon,\epsilon)\longrightarrow \wh{\cg}$  the corresponding Euler-Arnold  equation on $(\widehat{G},\ll  .~,~.  \gg_{\wh\cg})$ is given by
\begin{equation*}
\left\{ \begin{array}{ll}  \partial_tu  +  \ad^*_uu  -  a(t)Tu  =  0   \\
\partial_ta(t)=0
\end{array}\right.
\end{equation*}
where the second equation implies that $a(t)=a$ is constant.\\
%
%
%
%
\subsection{ Covariant derivative and the contraction operator}
In the sequel we recall the concept of covariant derivative induced on the central extension. 
According to \cite{Arnold-Khesin}, for $(X,a),(Y,b)\in\hat{\cg}$ the right invariant connection (covariant derivative) $\wh{\nabla}_{(X,a)}(Y,b)$ is given by
\begin{eqnarray}\label{covariant derivative}
2\wh{\nabla}_{(X,a)}(Y,b)  &=&  -\wh\ad_{(X,a)}(Y,b)  + \wh{\ad}^*_{(X,a)}(Y,b)+\wh{\ad}^*_{(Y,b)}(X,a)\nonumber\\
&=&   \Big(  [X,Y],\omega(X,Y) \Big)   +   \Big(  \ad^*_XY -b TX , 0 \Big) + \Big( \ad^*_YX -a TY , 0  \Big)  \nonumber\\
&=&  \Big(  [X,Y]    + \ad^*_XY  + ad^*_YX - b TX  -  a TY ,    \omega(X,Y)  \Big)\nonumber\\
&=&  \Big(  2\nabla_XY - ( b TX  +  a TY) ,    \omega(X,Y)  \Big)
\end{eqnarray}
where $2\nabla_XY=  -\ad_XY+ \ad^*_XY  + \ad^*_YX$ is the covariant derivative on $G$.
As a consequence we have
\begin{eqnarray}\label{second covariant derivative}
\wh{\nabla}_{(X,a)}\wh{\nabla}_{(Y,b)}(Z,c)  &=& \Big(  \nabla_X\nabla_YZ  - \frac{1}{2}\nabla_X(cTY+bTZ)  -\frac{1}{4}\omega(Y,Z)TX  \\
&&\hspace{-20mm}-\frac{1}{2}  aT \big(\nabla_ZY -\frac{1}{2} (cTY+bTZ ) \big)     
~,  \frac{1}{2}\omega \big(X,\nabla_ZY -\frac{1}{2} (cTY+bTZ  )  \big)  \Big).\nonumber
\end{eqnarray}
In the next proposition, we compute the counter part of the correction term
\begin{equation}\label{correction term K}
\wh{K}(\wh{u},\wh{X}):=  \frac{1}{2} \left( \wh{\nabla}_{\wh{X}}\wh{\nabla}_{\wh{X}}\wh{u} +\wh{R}(   \wh{u} ,\wh{X}) )\wh{X} \right)
\end{equation}
for the central extension. The term $\wh{K}$, introduced in \cite{A-C-C14}, theorem 3.2, appears when we compute the Euler-Arnold equation after considering the stochastic Lagrangian. 
%
%
%
%
\begin{prop}\label{prop contractio form}
\textbf{a.} For $\wh{X}=(X,0), \wh{u}=(u,a)\in\hat\cg$ the following holds true
\begin{eqnarray}
&&  \widehat{K}\big( \wh{u},\wh{X} \big)   
%
%
=\Big( K(u,X)  -\frac{1}{4}aT\nabla_XX    + \frac{1}{2}  \omega(u,X)TX    ~,~    \frac{1}{2}\omega(u,\nabla_XX)   \Big) 
\end{eqnarray}
where  $K(u,X)=\frac{1}{2}\left(  \nabla_X\nabla_Xu  + R(u,X)X \right).$ \\
\textbf{b.} In the case that $\nabla_XX=0$ we have
\begin{eqnarray*}
2\wh{K}\big((u,a),(X,0)\big)    =    \Big(   2K(u,X)  +   \omega(u,X)TX    ~,~ 0            \Big) .  
\end{eqnarray*}
\end{prop}
The proof is moved to the section Appendix.
%
%
%
%
%
%
%
%
\section{Stochastic Euler-Poincar\'e reduction for central extension}\label{Sec stoch EPR}
In this section, we extend the framework introduced in \cite{A-C-C14} to incorporate stochastic perturbations in the context of central extensions.

Let $G$ be a (finite-dimensional) Lie group and $\widehat{G}$ be its 1-dimensional central extension, with the Lie algebras $\mathfrak{g}$ and $\widehat{\mathfrak{g}}$, respectively. Moreover, suppose that the inner products $\ll\cdot,\cdot\gg_{\mathfrak{g}}$ and $\ll\cdot,\cdot\gg_{\widehat{\mathfrak{g}}}$ on $\mathfrak{g}$ and $\widehat{\mathfrak{g}}$, respectively, define right-invariant metrics on $G$ and $\widehat{G}$. For $\tau  > 0$, consider the curve 
\[
\gamma:[0,\tau] \longrightarrow \widehat{G};\quad t \longmapsto \gamma(t) 
\]
and let $\{\wh{H}_i:=(H_i,0)\}_{1 \leq i \leq m}$ be elements of $\widehat{\mathfrak{g}}$. Moreover suppose that $\wh{u}:=(u(t),a(t))$ be a deterministic $C^1$ map from $[0,\tau]$ to $\widehat{\mathfrak{g}}$. Consider the $\widehat{\mathfrak{g}}$-valued semi-martingale
\begin{eqnarray}\label{g valued semimartingale}
d\g(t)  &=&   T_eR_{\g(t)}  \Big(  \sum_{i=1}^m(H_i,0)\circ dW^i_t -\frac{1}{2}\widehat{\nabla}_{(H_i,0)}(H_i,0)dt +(u(t),a(t))dt   \Big)\nonumber \\
&  =&   T_eR_{\g(t)}  \Big(  \big( \sum_{i=1}^mH_i\circ dW^i_t -\frac{1}{2}\nabla_{H_i}H_idt +u(t)dt ~ , ~  
a(t) dt  \big)  \Big).
\end{eqnarray}
    Here $W_t$ is an $\R^m$-valued Brownian motion and $R$ denotes the right translation.\\
\begin{Remark}
In practice, we apply the noise to the particles of the  fluid on the manifold. The parameter $a$ appears as the effect of rotation (Coriolis effect) when we consider $G={\cD_{\vol }^s(M)}$ and its central extension (see e.g. \cite{Suri2023}).
\end{Remark}
Since the argument in \cite{A-C-C14} is valid for any Lie group, by the It\^{o} formula we set
\begin{equation*}
\frac{D_t^{\widehat{\nabla}}}{dt} \widehat{\g}(t)   =   T_eR_{\widehat{\g}(t)} (u(t),a(t)).
\end{equation*}
Consider the stochastic kinetic energy functional
\begin{equation}\label{stochastic energy functional}
J^{\widehat{\nabla}}(\xi) =  \frac{1}{2}  \E\Big[ \int_0^\tau  \ll  T_{\xi(t)} R_\xi(t)^{-1}  \frac{D_t^{\widehat{\nabla}}{ {  \xi  }(t)}}{dt}   , T_{\xi(t)}R_{ \xi(t)^{-1}}    \frac{D_t^{\widehat{\nabla}}{  \xi  (t)}}{dt}   \gg_{\wh\cg}      dt\Big]
\end{equation}
for any $\wh{G}$-valued semi-martingale $\xi$. We recall that $\E$ in \eqref{stochastic energy functional} denotes the expectation.

\begin{Remark} If $G$ is a finite dimensional Lie group then (\ref{g valued semimartingale}) always has a   solution  \cite{Ikeda-Watanabe}. If $G$ and $\wh{G}$ are  diffeomorphism group then, under suitable conditions on $\{\wh{H}_i\}$ and $\wh{u}$ we will try to prove that  the solution for the  equation (\ref{g valued semimartingale}) also exist.
\end{Remark}
Choose $\{H_i\}_{1 \leq i \leq m}\subseteq \cg$ with the property $\nabla_{H_i}H_i=0$ for all $1\leq i\leq m$. The following theorem is a modification of theorem 3.4 in \cite{A-C-C14} for the central extension setting. 
%
%
%
%
\begin{thm}
The $\wh{G}$-valued semi-martingale $\g(\cdot)$ defined by  \eqref{g valued semimartingale} is a critical point of $J^{\wh{\nabla}}$ if and only if the deterministic  path $\wh{u}=(u,a) \in C^{1}([0, \tau];\wh\cg )$ satisfies the following equations
\begin{eqnarray}\label{modified Euler Arnold}
\left\{ \begin{array}{ll}  
 \frac{d}{dt} {u}(t) =-  \ad_{u(t)}^{*} u(t) +a(t)Tu(t) +  K(u) +  \sum_{i}   \frac{1}{2}\omega(u,H_i)TH_i
 \\
  \frac{d}{dt} {a}(t)= 0 
\end{array}\right.
\end{eqnarray}
where the operator $K:  \cg \rightarrow  \cg$ is given by $K(u)= \frac{1}{2} \sum_i\nabla_{H_i} \nabla_{H_i}u +R(u,H_i)H_i$.
\end{thm}
\begin{proof}
According to \cite{A-C-C14}, theorem 3.4 the semi-martingale \eqref{g valued semimartingale} is a critical point of     \eqref{stochastic energy functional} if and only if the drift part $\wh{u}$ satisfies the modified equation
\begin{equation*}
\frac{d}{dt} \wh{u}(t) = - \wh{\ad}_{\tilde{u}(t)}^{*} \wh{u}(t) + \wh{K}(\wh{u}(t)) 
\end{equation*}
where $\tilde{u}=\wh{u} - \frac{1}{2}\sum_i  \wh\nabla_{\wh{H}_i}\wh{H}_i$. However, according to the \eqref{covariant derivative} we have
\begin{eqnarray*}
\tilde{u}  &=&  \wh{u} - \frac{1}{2}\sum_i  \wh\nabla_{\wh{H}_i}\wh{H}_i
=  \big(u,a\big) - \frac{1}{2}\sum_i  \big(  \nabla_{ {H}_i}{H}_i,0 \big)\\
&=& \big(u - \frac{1}{2}\sum_i   \nabla_{{H}_i}{H}_i ~,~  a\big) =  \big(u  ~,~  a\big)  =\wh{u}.
\end{eqnarray*}
Moreover using proposition \ref{prop contractio form} we get
\begin{eqnarray*}
\wh{K}(\wh{u})  &=& \frac{1}{2} \sum_i \Big(  \wh{\nabla}_{\wh{H}_i}\wh{\nabla}_{\wh{H}_i}\wh{u}  +\wh{R}(\wh{u},\wh{H}_i)\wh{H}_i         \Big)\\
&=&   \Big( K(u) ~,~0\Big) +  \sum_{i}\Big(  \frac{1}{4} a T(\nabla_{H_i}H_i) +  \frac{1}{2}\omega(u,H_i)TH_i  ~,~ \frac{1}{4}\omega(u,\nabla_{H_i}H_i)  \Big) \\
&=&   \Big( K(u) ~,~0\Big) +  \sum_{i}\Big(   \frac{1}{2}\omega(u,H_i)TH_i  ~,~ 0 \Big). 
\end{eqnarray*}
As a result
\begin{eqnarray*}
\frac{d}{dt} \wh{u}(t) &=&  - \wh{\ad}_{\tilde{u}(t)}^{*} \wh{u}(t) + \wh{K}(\wh{u}(t)) \\
&=&  - \wh{\ad}_{\wh{u}(t)}^{*} \wh{u}(t) + \wh{K}(\wh{u}(t)) \\
&=&- \Big( \ad_{u(t)}^{*} u(t) -a(t)Tu(t)~,0\Big) +  \Big( K(u) +  \sum_{i}   \frac{1}{2}\omega(u,H_i)TH_i  ~,~ 0 \Big) \\
&=&   \Big(  -\ad_{u(t)}^{*} u(t) + a(t)Tu(t) +  K(u) +  \sum_{i}   \frac{1}{2}\omega(u,H_i)TH_i  ~,~ 0 \Big) 
\end{eqnarray*}
which implies that  $\frac{d}{dt} {a}(t)=0$ i.e. $a$ is constant and
\begin{eqnarray*}
\frac{d}{dt} {u}(t) =   - \ad_{u(t)}^{*} u(t) +aTu(t) +  K(u) +  \sum_{i}   \frac{1}{2}\omega(u,H_i)TH_i  .
\end{eqnarray*}
\end{proof}
%
%
%
\begin{rem}
If we omit the condition $\nabla_{H_i}H_i=0$, then more correction terms in \eqref{modified Euler Arnold} would appear, and especially $\frac{d}{dt}a(t) \neq 0$. Intuitively, this means that the rate of rotation would vary if we choose improper directions for adding noise to the velocity field $u$.   
\end{rem}
%
%
%
%
%
\section{Viscous quasi-geostrophic equations on  torus}\label{Section Viscous QGS}
In this section $G=\cD_{\vol}^s(\T^2)$ and $\cg=T_eG$. Following   \cite{A-C-C14} and \cite{Cip-Cru}, we consider a basis for the space of Hamiltonian vector fields on torus generated by eigenfunctions of the Laplacian. More precisely for $k=(k_1,k_2)\in\mathbb{Z}^2$ and $\theta=(\theta_1,\theta_2)\in\T^2$ we consider the divergence free vector fields
\begin{eqnarray*}
&&A_k(\theta):= \lambda(|k|)\Big(   k_2\cos(k.\theta)  ~,~-k_1\cos(k.\theta)  \Big)=  - \lambda(|k|)\nabla^\perp\sin(k.\theta) ,\\
&&A_k(\theta):= \lambda(|k|)\Big(   k_2\sin(k.\theta)  ~,~-k_1\sin(k.\theta)  \Big)=  \lambda(|k|) \nabla^\perp\cos(k.\theta)
\end{eqnarray*}
where $k.\theta=k_1\theta_1+k_2\theta_2$,  $\lambda(|k|)=\lambda(|k_1|+|k_2|)$ is a function of $k$ and $\nabla^\perp f=(-\partial_{2}f,\partial_{1}f)$ is the Hamiltonian vector field associated with the differentiable function $f\in C^2(\T^2)$. For example in \cite{Cip-Cru},  $\lambda(|k|)=\frac{1}{|k|^r}$ where $r\geq 3.$ Specially, following theorem 2.2 from \cite{Cip-Cru} and also page 12 \cite{A-C-C14}, for any $f\in C^2(\T^2)$ we can write the estimate 
\begin{eqnarray}\label{Laplace approx}
\sum_{|k|\leq m} A_kA_kf+ B_kB_kf=\nu\Delta f 
\end{eqnarray}
where $\nu=\frac{1}{2}\sum_{|k|\leq m} \lambda^2(|k|)k_1^2 $}. When its is necessary, we can pass to this finite series.
For any $k$, define  $\wh{A_k}=(A_k,0)\in\wh\cg$ and  $\wh{B_k}=(B_k,0)\in\wh\cg$  where $\wh\cg=\cg\rtimes_{\omega_\alpha}\R$ and $\omega_\alpha$ is the Roger cocycle induced by th closed 1-form $\alpha=\beta\theta_2 \in\Omega^1( \T^2)$. Then we have 
\begin{eqnarray*}
\wh\nabla_{\wh{A_k}}{\wh{A_k}}= \big( \nabla_{{A_k}}{{A_k}} ~,~\omega_\alpha(A_k,A_k)  \big) = \big( \nabla_{{A_k}}{{A_k}} ~,~0  \big)  =  (0,0) 
\end{eqnarray*}
and similarly $\wh\nabla_{\wh{B_k}}{\wh{B_k}}= \big( \nabla_{{B_k}}{{B_k}} ~,~0  \big) =  (0,0) $. 
In order to  add the effect of rotation (Coriolis effect) to our setting, we will use the semi-martingale 
\begin{eqnarray}\label{hatG v2 valued semimartingale torus}
\left\{ \begin{array}{ll}  
 d\g(t) =      T_e\wh{R}_{\gamma(t)}\Big(  \sum_{k}  \big[ {\wh{A}_k}\circ dW^{k,1}_t +{\wh{B}_k} \circ dW^{k,2}_t \big]   + {\wh{u}}(t))dt  \Big)     \\
  \g(0) = id_{\wh{G}} 
\end{array}\right.
\end{eqnarray}
Here  $\wh{u}=(u,a)\in C^1(   [0,\tau],\wh\cg)$ and $\{  W^{k,1}_t ,  W^{k,2}_t\}_k$ are (independent) real valued Brownian motions.
%
%
%
%
%
%

Since $\omega_\alpha$ is defined by 
\[
\omega_\alpha(X_f,X_g)=\int_{\T^2}f\alpha(X_g)d\vol
\]
with $\alpha = \beta \theta_2$, for a suitable constant $\beta$, it is integrable and it  integrates to a group cocycle, say, $\Omega_\alpha$.
Following the discussion  in section \ref{section Roger Cocycle int},  the group  $\wh{G}$ 
in an identity neighborhood appears as  $\wh{U}=U\times \Sph^1\subseteq \wh{G}$ where $U\subseteq G$ is an identity neighborhood.
\begin{prop}\label{prop gen derivative}
For $\tau$ sufficiently small and $(u, a) \in C^1([0, \tau], \wh{\cg})$, the stochastic differential equation \eqref{hatG v2 valued semimartingale torus} admits a  solution $\g:[0,\tau]\times S \longrightarrow U\times\Sph^1\subseteq\wh{G}$  such that $\frac{D_t^{\widehat{\nabla}}\gamma(t)}{dt}  = (u(t), a)$.
\end{prop}
\begin{proof}
For the curve  $\g$ there exist $g:[0,\tau]\longrightarrow U$ and   $c:[0,\tau]\longrightarrow \Sph^1$ such that $\g(t)=(g(t),c(t))$. Consider the stochastic differential equation
\begin{eqnarray}\label{SDE V3}
d(g(t),c(t)) 
&=& T_{(e,1)}\wh{R}_{(g(t),c(t))} 
\Big(  \sum_{k}  \big[ {{A}_k} \circ dW^{k,1}_t +{{B}_k}  \circ dW^{k,2}_t \big]\\
&& + {{u}}(t))dt ~,~ adt  \Big)    
\end{eqnarray}
Using equation \eqref{Righ translation}, we observe that the first component of the above equation is
\begin{eqnarray*}
dg(t )  = T_{e}R_{(g(t)} 
\Big(  \sum_{k}  \big[ {{A}_k} \circ dW^{k,1}_t +{{B}_k}  \circ dW^{k,2}_t \big] + {{u}}(t))dt  \Big)   
\end{eqnarray*}
or equivalently for any $\theta\in\T^2$
\begin{eqnarray*}
g(t,\theta )  = \int_0^t \sum_{k}  \big[ {{A}_k}(g(s,\theta))\circ dW^{k,1}_t +{{B}_k}(g(s,\theta)) \circ dW^{k,2}_t \big]    + \int_0^t u(s,g(s,\theta))ds.     
\end{eqnarray*}
According to \cite{Cip-Cru} theorem 3.1 the above equation has a solution and depending on the regularity of $u$ and the regularity of the noise (or the exponent $r$ in $\lambda(|k|)=\frac{1}{|k|^r}$) the solution $g$ belongs to $C([0,\tau], 
 U)$. \\
On the other hand, once we have the solution $g$, the second component of \eqref{SDE V3} reads as follows:
\begin{eqnarray*}
c(t)  = \int_0^t e^{i(a +c(s) + \Omega_\alpha( g(s),g(s)))} ds
\end{eqnarray*}
where we used the exponential map to place the solution in the circle  $\Sph^1$. Continuity of $g$ and differentiability of $\Omega_\alpha$ ensure that \eqref{SDE V3} has a solution as long as $g$  remains in $U$. 
Finally, the computations on page 11 \cite{A-C-C14} and the fact that $\wh{\nabla}$ is right invariant imply that the  ${\widehat{\nabla}}$-generalized of the process \eqref{SDE V3} is $\frac{D_t^{\widehat{\nabla}}\gamma(t)}{dt} =(u(t),a)$.
\end{proof}
The resulting curve $\gamma$ for any $t \in [0, \tau]$ a.s. belongs to $U \times \Sph^1$. As a result, for any $\theta \in \T^2$, a.s. we have   $\gamma(t, \theta)\in \T^2 \times \Sph^1$.
%
%
%
%
%
\subsection{ The  $\wh{G}$ valued variation }\label{Section variation}
In order to define a variation, we consider the curve  in $U\times \Sph^1\subseteq\wh{G}$ given by the ordinary differential equation
\begin{eqnarray*}
\Big(\frac{d}{dt} e(\eps,t)\Big)e(\eps,t)^{-1}=\eps\big( \dot{v}(t)~  ,~\dot{b}(t)  \big)\quad;\quad e(\eps,0)=id_{\wh{G}}.
\end{eqnarray*}
where $\wh{v}(t)=(v(t),b(t))$ belongs to $C^1([0,\tau],\wh{\cg})$.  Then, we define the variation 
\begin{equation}\label{variation in hat G}
{\g}_\eps(t)=e(\eps,t)\circ \g(t).
\end{equation}
where $\gamma$ is the curve in proposition \ref{prop gen derivative}. To make the presentation more concise, we write 
\begin{eqnarray*}
\left\{ \begin{array}{ll}  
 d\g(t) =      T_e\wh{R}_{\gamma(t)}\Big(  \sum_{k}  \big[ {\wh{H}_k}\circ dW^{k}_t\big]   + {\wh{u}}(t))dt  \Big)     \\
  \g(0) = id_{\wh{G}} 
\end{array}\right.
\end{eqnarray*}
rather than \eqref{hatG v2 valued semimartingale torus}. By It\^{o} formula we get
\begin{eqnarray*}
d\g_\eps(t,\theta) &=&     \sum_{k}   T_{\g(t,\theta)} e\big(\eps,t,\g(t,\theta)\big) \wh{H}_k(\g(t,\theta)\circ dW^k_t  \\
&&+  T_{\g(t,\theta)} e\big(\eps,t,\g(t,\theta)\big) \wh{u}(t,\g(t,\theta))dt      +\eps \dot{\hat{v}}\big(t,\g_\eps(t,\theta)\big)dt   \\
&=&    \sum_{k}   \big(\wh{\Ad}_{e(\eps,t)}  \wh{H}_k\big) (\g_\eps(t,\theta) )\circ dW^k_t  \\
&&+  \big(\wh{\Ad}_{e(\eps,t)} \wh{u}(t,\g_\eps(t,\theta))dt      +\eps \dot{\hat{v}}\big(t,\wh\g_\eps(t,\theta)\big)  dt
\end{eqnarray*}
where $\dot{\hat{v}}(t)=\big( \dot{v}(t)~  ,~\dot{b}(t)  \big)$. 
Setting $\wh{H}_k^\eps(t)  =  \wh{\Ad}_{e(\eps,t)}\wh{H}_k$ and 
\[
\wh{\nabla}_{\wh{H}_k } \wh{H}_k=({\nabla}_{{H}_k } {H}_k,\omega(H_k,H_k))=0
\]
 we obtain 
\begin{eqnarray}\label{Eq Gen Der Var}
&&T_{\g_\eps(t)}R_{\g_\eps(t)^{-1}} \frac{D_t^{\widehat{\nabla}}}{dt} \g_{\eps}(t) \nonumber\\
&=&        \sum_{k}  \frac{1}{2} \wh{\nabla}_{\wh{H}_k^\eps(t)} \wh{H}_k^\eps(t)    
+
\wh{\Ad}_{e(\eps,t)} \big( -\frac{1}{2} \wh{\nabla}_{\wh{H}_k } \wh{H}_k+ \wh{u}(t)   \big)   + \eps \dot{\hat{v}}(t) \nonumber\\
%
%
%
%
%
%
 &=&   \sum_{k}  \frac{1}{2} \wh{\nabla}_{\wh{H}_k^\eps(t)} \wh{H}_k^\eps(t)    
+
\wh{\Ad}_{e(\eps,t)}   \wh{u}(t)      + \eps \dot{\hat{v}}(t).
\end{eqnarray}
Moreover we have,
\begin{eqnarray*}
&&\frac{d}{d\eps}|_{\eps=0}  T_{\g_\eps(t)}R_{\g_\eps(t)^{-1}} \frac{D_t^{\widehat{\nabla}}}{dt} \g_{\eps}(t)\\
&=&       \frac{1}{2} \sum_{k}   \frac{d}{d\eps}|_{\eps=0}  \wh{\nabla}_{\wh{H}_k^\eps(t)} \wh{H}_k^\eps(t)    
+
\frac{d}{d\eps}|_{\eps=0}   \wh{\Ad}_{e(\eps,t)}   \wh{u}(t)      +  \dot{\hat{v}}(t)\\
&=&        \frac{1}{2}\sum_{k}   \Big(   \wh{\nabla}_{\wh{\ad}_{\wh{v}}\wh{H}_k} \wh{H}_k   +  \wh{\nabla}_{\wh{H}_k}( \wh{\ad}_{\wh{v}}\wh{H}_k )  \Big)
+  \wh{\ad}_{\wh{v}(t)}  \wh{u}(t)      +  \dot{\hat{v}}(t).
\end{eqnarray*}
Using \eqref{covariant derivative}  we get
\begin{eqnarray*}
\wh{\nabla}_{\wh{\ad}_{\wh{v}}\wh{H}_k} \wh{H}_k    &=&     \wh{\nabla}_{(\ad_vH_k , -\omega(v,H_k))} (H_k,0)\\
&=& \Big(  {\nabla}_{\ad_vH_k } H_k + \frac{1}{2} \omega(v,H_k) TH_k ~ ,  \frac{1}{2}\omega( \ad_vH_k,H_k) \Big)
\end{eqnarray*}
and similarly
\begin{eqnarray*}
 \wh{\nabla}_{\wh{H}_k}( \wh{\ad}_{\wh{v}}\wh{H}_k )    &=&     \wh{\nabla}_{(H_k,0)} (\ad_vH_k , -\omega(v,H_k))  \\
&=& \Big(  {\nabla}_{H_k} \ad_vH_k   + \frac{1}{2}  \omega(v,H_k) TH_k ~ ,  \frac{1}{2} \omega( H_k , \ad_vH_k) \Big)
\end{eqnarray*}
The last two equations  imply that 
\begin{eqnarray*}
&&  \wh{\nabla}_{\wh{\ad}_{\wh{v}}\wh{H}_k} \wh{H}_k   +  \wh{\nabla}_{\wh{H}_k}( \wh{\ad}_{\wh{v}}\wh{H}_k ) \\
&=&     \Big(  {\nabla}_{\ad_vH_k } H_k + \frac{1}{2} \omega(v,H_k) TH_k ~ ,  \frac{1}{2}\omega( \ad_vH_k,H_k) \Big)\\
&& +  \Big(  {\nabla}_{H_k} \ad_vH_k   + \frac{1}{2}  \omega(v,H_k) TH_k ~ ,  \frac{1}{2} \omega( H_k , \ad_vH_k) \Big)\\
&=&     \Big(  {\nabla}_{\ad_vH_k } H_k +{\nabla}_{H_k} \ad_vH_k  +  \omega(v,H_k) TH_k ~ , ~ 0 \Big).
\end{eqnarray*}
As a result
\begin{eqnarray}\label{Eq Gen Der Var2}
&&\frac{d}{d\eps}|_{\eps=0}  T_{\g_\eps(t)}R_{\g_\eps(t)^{-1}} \frac{D_t^{\widehat{\nabla}}}{dt} \g_{\eps}(t)\\
&=&       \frac{1}{2}  \sum_{k}  \Big(  {\nabla}_{\ad_vH_k } H_k +{\nabla}_{H_k} \ad_vH_k  +  \omega(v,H_k) TH_k ~ , ~ 0 \Big) 
+ \wh{\ad}_{\wh{v}(t)}  \wh{u}(t)      +  \dot{\hat{v}}(t)\nonumber.
\end{eqnarray}
\subsection{ Viscous QGS on $\T^2$ with linear damping}
%
%
%
%
%
%
%
%
Suppose that $u=\nabla^\perp \psi_u$ where $\psi_u:\T^2\longrightarrow \R$, $\alpha=\beta\theta_2$, $d\theta:=\theta_1\wedge \theta_2$ and $\int_{\T^2}d\theta=N$. We recall that the (linear) operator $T:\cg\longrightarrow \cg$ is given by $\omega_\alpha (u,v)=\ll Tu,v\gg_\cg$ (see equation \eqref{Operator T}).
\begin{la}\label{lemma sigma}
For any $k\in\mathbb{Z}^2$, the followings hold true.\\
\textbf{i.} $TA_k =  \beta k_1\lambda(|k|)^2N B_k$, \\
\textbf{ii.} $TB_k=-\beta k_1 \lambda(|k|)^2N A_k$, \\
\textbf{iii.} $\omega(u,A_k) = -\beta k_1\lambda(|k|)  (\psi_u,\cos k.\theta)_{L^2}$, \\
\textbf{iv.} $\omega(u,B_k) = -\beta k_1 \lambda(|k|)(\psi_u,\sin k.\theta)_{L^2}$,\\
\textbf{v.} $\sum_k\omega(u,A_k)TA_k+\omega(u,B_k)TB_k=-\beta^2N u :=-2\sigma u.$
\end{la}
The proof is moved to the section Appendix.\\[2.3mm]
%
%
%
%
%
%
Let $\wh{u}=(u,a)\in\wh{\cg}$ where  $u\in C^1([0,\tau],\cg)$ and $a\in\R$ is a constant. 
\begin{thm}\label{thm vis and fric}
The process \eqref{hatG v2 valued semimartingale torus} is a critical point of the functional     
$J^{\widehat{\nabla}}$ if and only if $\wh{u}$ satisfies the following equation for some function  $p$ (know as the pressure).
\begin{equation*}
\left\{ \begin{array}{ll}  \partial_tu  +  (u.\nabla)u  -  aTu   - {\nu}\Delta u + \sigma u    =  \mathrm{grad} p   \\
\mathrm{div} u=0 
\end{array}\right.
\end{equation*}
\end{thm}
%
%
\begin{proof}
For the specified   semi-martingale \eqref{hatG v2 valued semimartingale torus} using equation \eqref{Eq Gen Der Var} we get
\begin{eqnarray*}
T_{\g_\eps(t)}R_{\g_\eps(t)^{-1}} \frac{D_t^{\widehat{\nabla}}}{dt} \g_{\eps}(t) &=&               \frac{1}{2}  \sum_{k} \big( \wh\nabla_{{\wh A}_k^\eps(t)} {\wh A}_k^\eps(t)   +   \wh\nabla_{{\wh B}_k^\eps(t)} {\wh B}_k^\eps(t)  \big)  \\
&&   +   \wh{\Ad}_{e(\eps,t)} \wh{u}(t)      + \eps \big( \dot{v}(t)  ,\dot{b}(t) \big)
\end{eqnarray*}
where 
\[
{\wh A}_k^\eps(t)  :=   \wh{\Ad}_{e(\eps,t)}\wh{A_k} =\wh{\Ad}_{e(\eps,t)}(A_k,0)
\]
and similarly 
\[
{\wh B}_k^\eps(t)  :=  \wh{\Ad}_{e(\eps,t)}\wh{B_k}  =\wh{\Ad}_{e(\eps,t)}(B_k,0).
\]
Note that 
\[
\frac{d}{d\eps}|_{\eps=0} {\wh A}_k^\eps(t)  =    \wh{\ad}_{\wh{v}(t)}\wh{A_k} =  \big(\ad_{v(t)}A_k ~,-\omega(v(t),A_k) \big)
\]
and  $\frac{d}{d\eps}|_{\eps=0} {\wh B}_k^\eps(t)  =     \big(\ad_{v(t)}B_k ~,-\omega(v(t),B_k) \big)$.
As a result, equation \eqref{Eq Gen Der Var2}  implies that 
\begin{eqnarray*}
&&\frac{d}{d\eps}|_{\eps=0} T_{\g_\eps(t)}R_{\g_\eps(t)^{-1}} \frac{D_t^{\widehat{\nabla}}}{dt} \g_{\eps}(t) \\
&=&    \frac{1}{2}   \sum_{k} \Bigg(  \Big(  {\nabla}_{\ad_vA_k } A_k +{\nabla}_{A_k} \ad_vA_k  +  \omega(v,A_k) TA_k ~ , ~ 0 \Big)\\
&&  +   \Big(  {\nabla}_{\ad_vB_k } B_k +{\nabla}_{B_k} \ad_vB_k  +  \omega(v,B_k) TB_k ~ , ~ 0 \Big)   \Bigg)\\
&&+ \wh{\ad}_{\wh{v}(t)}  \wh{u}(t)      +  \dot{\hat{v}}(t).
\end{eqnarray*}%
Since the Riemann curvature tensor on $\T^2$ vanishes, then
\[
{\nabla}_{\ad_vA_k } A_k +{\nabla}_{A_k} \ad_vA_k = {\nabla}_{A_k } {\nabla}_{A_k } v
\]
and 
\[
 {\nabla}_{\ad_vB_k } B_k +{\nabla}_{B_k} \ad_vB_k  =  {\nabla}_{B_k} {\nabla}_{B_k} v.
\]
Accordingly, we obtain
\begin{eqnarray*}
&&\frac{d}{d\eps}|_{\eps=0} T_{\g_\eps(t)}R_{\g_\eps(t)^{-1}} \frac{D_t^{\widehat{\nabla}}}{dt} \g_{\eps}(t) \\
&=&    \frac{1}{2}   \sum_{k}   \Big(  {\nabla}_{A_k } {\nabla}_{A_k } v  +  {\nabla}_{B_k} {\nabla}_{B_k} v ~ , ~ 0 \Big)\\
&&    \frac{1}{2}   \sum_{k}   \Big(   \omega(v,A_k) TA_k  +  \omega(v,B_k) TB_k  ~ , ~ 0 \Big)   \\
&&  + \wh{\ad}_{\wh{v}(t)}  \wh{u}(t)      +  \dot{\hat{v}}(t)\\
&\stackrel{*}{=}&             \Big( \nu\Delta v - \sigma v ~,~ 0  \Big)   
+   \wh{\ad}_{\wh{v}} \wh{u}(t)      +  \dot{\hat{v}}(t) 
\end{eqnarray*}
where in * we used  equation \eqref{Laplace approx} is used and lemma \ref{lemma sigma}.\\
We recall that $\ll~,~\gg_{\wh\cg}$, $\ll~,~\gg_{\cg}$ and $\langle~,~\rangle $ are the metrics on $\wh\cg$, $\cg$ and $\T^2$ respectively. Moreover, by assumption, $\wh{v}(0) = \wh{v}(\tau) = 0$. Consequently, by applying integration by parts and utilizing the fact that $\Delta$ is self-adjoint, we obtain
\begin{eqnarray*}
&&\frac{d}{d\eps}|_{\eps=0} J^{\widehat{\nabla}}(  \g_{\eps}(t)  )\\
&=&   \E  \int_{0}^{\tau} \ll 
  \dot{\wh{v}} (t)   + \wh{\ad}_{\wh{v}(t)}\wh{u}(t) + (  \nu\Delta v(t) - \sigma v(t)   ,0)  ~,~ \wh{u}(t) \gg_{\wh\cg}   dt\\
&=&    \int_{0}^{\tau}   \Big(   \ll  
  \dot{\wh{v}} (t)  , \wh{u}(t) \gg_{ \wh\cg}   +   \ll \wh{\ad}_{\wh{v}(t)}\wh{u}(t)  , \wh{u}(t) \gg_{ \wh\cg}\\
&&  +  \ll (  \nu\Delta v(t) -\sigma v(t)   ,0)  ~,~ \wh{u}(t) \gg_{\wh\cg}   \Big)dt
\end{eqnarray*}
\begin{eqnarray*}
&=&  \int_0^\tau  \ll - \dot{\wh{u}} (t)   - \wh{\ad}^*_{\wh{u}(t)}\wh{u}(t) + ( \nu\Delta u(t) - \sigma u(t) ,0)  ~,~ \wh{v}(t) \gg_{\wh{\cg}}  dt\\
 %
%
%
&=&  - \int_0^\tau  \int_{\T^2}  \Big\langle  
   \dot{u} (t)    +   {\ad}^*_{{u}(t)}{u}(t) - a T(u(t)) -  \nu\Delta u(t) +\sigma u(t)   ~,~ {v}(t)  \Big\rangle    d\theta dt.
%
%
\end{eqnarray*}
Since the above argument holds for any   $\wh{v}\in\wh\cg$ we have
\begin{equation*}
\left\{ \begin{array}{ll}  \partial_tu   +   {\ad}^*_{{u}(t)}{u}(t) - aT(u(t)) - \nu\Delta u +\sigma u   =  0   \\
\mathrm{div}(u)=0 
\end{array}\right.
\end{equation*}
Due to the fact that $u$ is divergence free and 
\[
{\ad}^*_{{u}(t)}{u}(t)=P_e(\nabla_uu)=\nabla_uu-\mathrm{grad}p=(u.\nabla)u  - \mathrm{grad}p
\]
the previous equations are equivalent to \eqref{Viscouse rot Euler}. We remind that the projection map $P_e:\cX(\T^2)\longrightarrow\mathrm{div}^{-1}(0)$ is induced by the Hodge decomposition and sends any vector field to its divergence free part \cite{Marsden-Ebin-Fischer}.
\end{proof}
%
%
%
%
%
\subsection{ Evolution  equation according to the stream functions }\label{sec QGS Stream fv}
The vorticity form of equation \eqref{Viscouse rot Euler} is derived in the following way.
First we write equation \eqref{Viscouse rot Euler} in terms of forms i.e. we apply the musical isomorphism $\flat$ induced by the Riemannian metric $\langle ., .\rangle$ on $\T^2$  
\begin{equation}\label{Viscous rot Euler v2}
  \partial_tu^\flat   +   (\nabla_uu)^\flat  + a (Tu)^\flat - \nu\Delta u^\flat + \sigma u^\flat   =  dp   \\
\end{equation}
Now apply the exterior derivative $d$ on both sides. Since  $d$ commutes with $\partial_t$ and the Laplacian $\Delta$ then
\begin{equation}\label{Viscous rot Euler v3}
  \partial_td u^\flat   +  a(\nabla_uu)^\flat   + a d(Tu)^\flat - \nu\Delta du^\flat +\sigma du^\flat   =  ddp=0 .
\end{equation}
The $2$-form $du^\flat$  on the 2 dimensional manifold $\T^2$  is referred to as vorticity.  Since that space of 2-forms on $\T^2$ is 1-dimensional, there exists a function $\phi:\T^2\longrightarrow \R$  such that $du^\flat=\phi d\theta$ where $d\theta=\theta_1\wedge\theta_2$. If $u=\nabla^\perp \psi$ then $\phi=\Delta\psi$. More precisely
\begin{eqnarray*}
du^\flat  =  d(- \partial_2\psi \theta_1 + \partial_1\psi  \theta_2) =   ( \partial_1^2\psi +  \partial_2^2\psi)  \theta_1\wedge  \theta_2  =  (\Delta \psi)  d\theta
\end{eqnarray*}
Moreover, according to theorem 1.17 chapter IV of \cite{Arnold-Khesin}, we have 
\begin{eqnarray*}
&& (\nabla_u u)^\flat = \cL_uu^\flat - \frac{1}{2} d\langle u, u\rangle
\end{eqnarray*}
where $\cL$ denotes the Lie derivative. As a result 
\begin{eqnarray*}
&& d(\nabla_u u)^\flat = d\cL_uu^\flat - \frac{1}{2} dd\langle u, u\rangle =d\cL_uu^\flat = \cL_u du^\flat =\cL_u (\Delta \psi d\theta).
\end{eqnarray*}
On the other hand
\begin{eqnarray*}
\cL_u \Delta \psi d\theta   &=&  d i_u (\Delta\psi d\theta) + i_u .d (\Delta\psi d\theta)  = d i_u (\Delta\psi d\theta)\\
&=&  d( \Delta \psi i_u  d\theta)  =  d( \Delta \psi  (-u_2 \theta_1 +  u_1  \theta_2)) \\
&=&    d( \Delta \psi  (-\partial_1\psi  \theta_1 - \partial_2\psi   \theta_2)) \\
&=&    d(   -\Delta \psi\partial_1\psi  \theta_1 - \Delta \psi\partial_2\psi   \theta_2   ) \\
&=&       -\partial_2(\Delta \psi\partial_1\psi) \theta_2\wedge  \theta_1 - \partial_1(\Delta \psi\partial_2\psi) \theta_1\wedge   \theta_2    \\
&=&     (  \partial_1\psi \partial_2\Delta \psi   - \partial_1 \Delta \psi  \partial_2\psi   ) \theta_1\wedge   \theta_2    \\
&=&   \{\psi,\Delta\psi\}   \theta_1\wedge   \theta_2    = \{\psi,\Delta\psi\} d \theta 
\end{eqnarray*}
where $u=(u_1,u_2)=(-\partial_2\psi,\partial_1\psi) $. Moreover for any $\psi,\xi\in C^\infty(\T^2)$,   $\{\psi,\xi\}=\partial_1\psi \partial_2\xi  -  \partial_2\psi \partial_1\xi$ is the Poisson bracket induced by the symplectic form $d\theta$ on $\T^2$.\\
Finally, we note that 
\begin{eqnarray*}
d(Tu)^\flat  = \beta \partial_1\psi  d\theta
\end{eqnarray*}
More precisely according to corollary 1 from \cite{Viz2}, $(Tu)  = P_e(\psi\alpha^\sharp)$ 
where $\sharp$ is the inverse of the musical isomorphism $\flat$ and $\alpha=\beta d\theta_2$. This last means that
\begin{eqnarray*}
d(Tu)^\flat  = d(\beta \psi  \theta_2)=\beta \partial_1\psi  \theta_1\wedge  \theta_2.
\end{eqnarray*}
As a result, for $a=1$,  equation \eqref{Viscous rot Euler v3} can be written as
\begin{eqnarray*}
\Big(  \partial_t\Delta\psi+  \{\psi , \Delta\psi\}  +  \beta \partial_1\psi - \nu\Delta^2 \psi  + \sigma \Delta\psi \Big) d\theta  =0
\end{eqnarray*}
or equivalently
\begin{equation*}
\partial_t\Delta\psi+  \{\psi, \Delta\psi \}   +  \beta\partial_1\psi   - \nu\Delta^2 \psi  + \sigma \Delta\psi=0
\end{equation*}
In the case of quasi-geostrophic equations on the sphere, the term $\nu\Delta^2\psi$ represents the turbulent viscosity and $\sigma \Delta\psi$ is Rayleigh friction (\cite{Skiba}, chapter 3 or \cite{Dymnikov} chapter 3).   
%
%
%
\subsection{Viscous QGS on $\T^2$ without Rayleigh friction}\label{sec new example sigma=0}
In this part, we propose a $\wh{\cg}$-valued diffusion process  using two vector fields which span the tangent space  of the torus at any point (not the Lie algebra $\cg$). We explore how this  semi-martingale can describe viscous QGS on $\T^2$ without Rayleigh friction via stochastic Euler-Poincar\'e reduction for central extension.

For any point $p\in \T^2$ set $H_1(p):=(\sqrt{2\nu},0)$ and $H_2(p)=(0,\sqrt{2\nu})$ where $\nu>0$ is a constant. Then $H_1,H_2$ belong to $\cg:=T_e\cD_{\vol}(\T^2)$. Consider the diffusion process 
\begin{eqnarray}\label{hatg valued semimartingale 1,2}
\left\{ \begin{array}{ll}  
 d\eta(t) =      T_e\wh{R}_{\eta(t)}\Big(  \sum_{i=1}^2  \big( {\wh{H}_i}\circ dW^{i}_t \big)   + {\wh{u}}(t)dt  \Big)     \\
  \eta(0) = id_{\wh{G}} 
\end{array}\right.
\end{eqnarray}
where $\wh{H}_i=(H_i,0)$, $i=1,2$ and 
\[
\wh{u}(t):[0,\tau]\longrightarrow\wh{\cg} \quad ; \quad t\longmapsto (u(t),a)
\]
is a $C^1$ map. The same argument as in Proposition \ref{prop gen derivative} shows that the process \eqref{hatg valued semimartingale 1,2} has a solution and its generalized derivative is $\wh{u}$.
Note that by assumption 
\[
\omega_\alpha(H_1,X_f) = \omega_\alpha(H_2,X_f) =\omega_\alpha(H_1,H_2) =0
\]
for any $X_f\in \cg$ (see remark \ref{remark kirillov}). As in section \ref{Section variation}, for  $\wh{v}(t)=(v(t),b(t))\in C^1([0,\tau],\wh{\cg})$ with $\wh{v}(0)= \wh{v}(\tau)=0 $ we set
\begin{equation*}
{\eta }_\eps(t)=e(\eps,t)\circ \eta(t).
\end{equation*}
where $\eta$ is a solution of \eqref{hatg valued semimartingale 1,2} and    $e(\eps,t)$ is given by \eqref{variation in hat G}. Then repeating the computations of section section \ref{Section variation}
for $\eta_\eps(t)$ we get
\begin{eqnarray}\label{Eq Gen Der Var 1,2}
&&\frac{d}{d\eps}|_{\eps=0}  T_{\eta_\eps(t)}R_{\eta_\eps(t)^{-1}} \frac{D_t^{\widehat{\nabla}}}{dt} \eta_{\eps}(t)\\
&=&       \frac{1}{2}  \sum_{i=1}^2  \Big(  {\nabla}_{\ad_vH_i } H_i +{\nabla}_{H_i} \ad_vH_i  ~ , ~ 0 \Big) 
+ \wh{\ad}_{\wh{v}(t)}  \wh{u}(t)      +  \dot{\hat{v}}(t)\nonumber\\
&=& \frac{1}{2}   \sum_{i=1}^2   \Big(  {\nabla}_{H_i }{\nabla}_{H_i } v  ~ , ~ 0 \Big)
+ \wh{\ad}_{\wh{v}(t)}  \wh{u}(t)      +  \dot{\hat{v}}(t)\nonumber\\
&=&             \Big( \nu\Delta v  ~,~ 0  \Big)   
+   \wh{\ad}_{\wh{v}} \wh{u}(t)      +  \dot{\hat{v}}(t) 
\end{eqnarray}
In the following theorem, we demonstrate that replacing the process \eqref{hatG v2 valued semimartingale torus} with \eqref{hatg valued semimartingale 1,2} results in the disappearance of the Rayleigh friction (linear damping) term from the evolution equation.
%
%
%
%
%
\begin{thm}
The process \eqref{hatg valued semimartingale 1,2} is a critical point of the functional     
$J^{\widehat{\nabla}}$ if and only if ${u}$ satisfies the following  equations.
\begin{equation*}
\left\{ \begin{array}{ll}  \partial_tu  +  (u.\nabla)u  -  aTu   - {\nu}\Delta u     =  \mathrm{grad} p   \\
\mathrm{div} u=0 
\end{array}\right.
\end{equation*}
\end{thm}
%
%
\begin{proof}
Following the proof of Theorem \ref{thm vis and fric} with the semi-martingale \eqref{hatg valued semimartingale 1,2} and the variation $\eta_\eps(t)$, we obtain
\begin{eqnarray*}
&&\frac{d}{d\eps}|_{\eps=0} J^{\widehat{\nabla}}(  \eta_{\eps}(t)  ) 
=
\E\Big[ \int_0^\tau  \ll \frac{d}{d\eps}|_{\eps=0} T_{\eta_\eps(t)}R_{\eta_\eps(t)^{-1} }  \frac{D_t^{\widehat{\nabla}}{ {  \eta_\eps  }(t)}}{dt}   ~,~ \wh{u}(t)   \gg_{\wh\cg}      dt\Big]\\
&=&   \E  \int_{0}^{\tau} \ll 
  \dot{\wh{v}} (t)   + \wh{\ad}_{\wh{v}(t)}\wh{u}(t) + (  \nu\Delta v(t)    , 0)  ~,~ \wh{u}(t) \gg_{\wh\cg}   dt \\
&=&    \int_{0}^{\tau}   \Big(   \ll  
  \dot{\wh{v}} (t)  ~,~ \wh{u}(t) +  \wh{\ad}_{\wh{v}(t)}\wh{u}(t)  \gg_{ \wh\cg}+  \ll (  \nu\Delta v(t)    , 0)  ~,~ \wh{u}(t) \gg_{\wh\cg}   \Big)dt  \\
&=&  - \int_0^\tau  \int_{\T^2}  \Big\langle  
   \dot{u} (t)    +   {\ad}^*_{{u}(t)}{u}(t) - a T(u(t)) -  \nu\Delta u(t)    ~,~ {v}(t)  \Big\rangle    d\theta dt
=0.
\end{eqnarray*}
Since in the above argument   $\wh{v}\in\wh\cg$ was arbitrary we have
\begin{equation*}
\left\{ \begin{array}{ll}  \partial_tu   +  (u.\nabla)u   - aTu - \nu\Delta u    =   \mathrm{grad}p \\
\mathrm{div}(u)=0 
\end{array}\right.
\end{equation*}
which completes the proof.
\end{proof}
Now, suppose that the velocity field $u$ is given by the stream function $\psi$. Then, using the results from section \ref{sec QGS Stream fv}, we can express the previous equation as
\begin{equation*}
\partial_t\Delta\psi+  \{\psi, \Delta\psi \}   +  \beta\partial_1\psi   - \nu\Delta^2 \psi  =0.
\end{equation*}
\textbf{Discussion.}
The two semi-martingales \eqref{hatG v2 valued semimartingale torus} and \eqref{hatg valued semimartingale 1,2} that we are considering lead to different physical phenomena due to their distinct structures.

The first semi-martingale \eqref{hatG v2 valued semimartingale torus} uses a basis for the Lie algebra $\cg$ to incorporate real-valued Brownian motions, which introduce stochastic perturbations to the system. This perturbation via the Coriolis effect manifests as viscosity and friction (linear damping) in the resulting evolution equation \eqref{Viscouse rot Euler}. Essentially, this randomness in the system induces energy dissipation, represented by the turbulent viscosity term $\nu\Delta u$ and the Rayleigh friction term $\sigma u$. In this scenario, setting $\sigma = 0$ would imply $\beta = 0$.

On the other hand, the second semi-martingale \eqref{hatg valued semimartingale 1,2} is constructed using a simpler set of stochastic terms; specifically, we use just two fixed vector fields for the fluctuation directions. This more constrained form of randomness does not generate the same kind of energy dissipation. As a result, the evolution equation derived from this semi-martingale lacks the linear damping terms. This reflects a physical situation where the system experiences stochastic perturbations,  viscosity, and the Coriolis effect but does not lose energy through linear damping.

In the absence of the central extension, both semi-martingales (after reducing to $\cg$) lead to the same evolution equation, which is the Navier-Stokes equations for incompressible fluids.
%
%
%
\section{Appendix}
\textbf{Proof of  proposition \ref{prop Roger cocycle is int}}.\\
\begin{proof}
Let  $N$ be the closed 1-dimensional submanifold of $\T^2$ which contains  the image of $\Phi(t,s_0)$. In   figure \eqref{fig:intersecting-circles},  the submanifold $N$ is denoted by a red circle with a thick line.
%
\begin{figure}[h]
    \centering
\begin{tikzpicture}[xscale=1.925]
    \draw[double distance=15mm] (0:1) arc (0:180:1);
    \draw[double distance=15mm] (180:1) arc (180:360:1);
    \draw[thick,white] (.6,0)--(1.37,0);
    \draw[thick,white] (-.6,0)--(-1.37,0);
    \draw[line width=1.3, looseness=.3, name path=mer, red] (0,-1.75cm-.2pt) to[out=0,in=0] (0,-.25cm+.2pt);
    \draw[line width=1.3, dashed, looseness=.3, red] (0,-1.75cm-.2pt) to[out=180,in=180] (0,-.25cm+.2pt);
    \draw[line width=.3, looseness=1.3, name path=lon, blue] (1.375cm+.1pt,0) to[out=-90, in=-90] (-1.375cm-.1pt,0);
    \draw[line width=.3, dashed, looseness=1.3, blue] (1.375cm+.1pt,0) to[out=90, in=90] (-1.375cm-.1pt,0);
    \fill[name intersections={of=mer and lon, by=v1}] (v1) ellipse (.2mm and .4mm);
    \node[above] at (v1) {~~~N};
\end{tikzpicture}
\caption{The submanifold $N$ for the singular cocycle $\omega_N$.}
    \label{fig:intersecting-circles}
\end{figure}
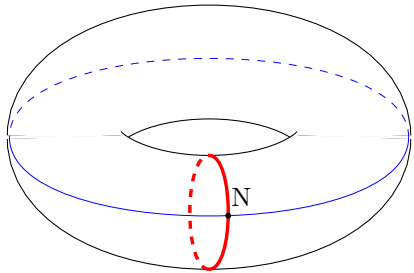

According to   \cite{Jan-Viz-2016} page 5035, $\omega_N$ and $\omega_\alpha$ are cohomologous if and only if   for any $\gamma\in\Omega^1(\T^2)$ with the property $d\gamma=0$ the equality 
\begin{equation}\label{Int-cond}
\int_N\gamma = \int_{\T^2}\alpha\wedge \gamma
\end{equation}
holds true. To make the proof  more clear, we will do this task in four steps.\vspace{4mm}\\
%
%
\noindent
\textbf{Step 1.} For $\gamma=\theta_1$ we have $\int_N\theta_1= \int_0^{2\pi}dt=2\pi$. On the other hand
\begin{eqnarray*}
\int_{\T^2}\alpha\wedge\gamma =   \frac{1}{2\pi}\int_{\T^2}\theta_1\wedge \theta_2  = \frac{1}{2\pi} \int_0^{2\pi}\int_0^{2\pi}dtds= 2\pi. 
\end{eqnarray*}
This means that \eqref{Int-cond} holds true.\vspace{4mm}\\
%
%
\textbf{Step 2.} For the constants $c_1,c_2\in\R$ and $\gamma=c_1 \theta_1 + c_2\theta_2 $ we have 
\begin{eqnarray*}
\int_{\T^2}\alpha\wedge \gamma =   \frac{c_1}{2\pi}\int_{\T^2}\theta_1\wedge \theta_2 + \frac{c_2}{2\pi}\int_{\T^2}\theta_2\wedge \theta_2   =  2c_1\pi. 
\end{eqnarray*}
On the other hand
\[
\int_N\gamma  =  c_1\int_N\theta_1 + c_2\int_N\theta_2 = 2c_1\pi + c_2\int_N\theta_2 .
\]
If we prove that $\int_N\theta_2=0$, then the assertion is  proved. 
Let 
\[
i:N\hookrightarrow \T^2\quad p\longmapsto (p,s_0)
\]
be the inclusion map which embeds $N$ into  $\T^2$. By restriction of $\theta_2 $ to $N$ basically we mean $i^*\theta_2$. Moreover, any tangent vector $v\in T_pN$ is  identified  with the tangent vector $(v,0)\in T_{(p,s_0)} \T^2$ via $i$.  As a result
\begin{eqnarray*}
i^*\theta_2(v) = \theta (d_piv)=\theta_2(v,0)_{(p,s_0)}    = 0
\end{eqnarray*}
which implies that the integral $\int_N\theta_2=0$.\vspace{4mm}\\
%
%
\textbf{Step 3.} For any $f\in C^\infty(\T^2)$    and $\gamma= f \theta_1  \in\Omega^1(\T^2)$ we have 
\begin{eqnarray*}
\int_{\T^2}\alpha\wedge \gamma =   \frac{1}{2\pi}\int_{\T^2} f \theta_1\wedge \theta_2  
\end{eqnarray*}
Note that  the condition $ d \gamma=0$ reads
\begin{eqnarray*}
d \gamma  = d( f \theta_1) = \partial_2f \theta_2\wedge\theta_1 =0
\end{eqnarray*}
which implies that $   \partial_2 f =0$. As a result, there is a  function $k=k(t)$ such that 
\[
\Phi^*f(t,s)=k(t).
\]
Using this fact we get
\[
\int_Nf\theta_1=\int_0^{2\pi} k(t) dt.
\]
On the other hand
\begin{eqnarray*}
\frac{1}{2\pi} \int_{\T^2}f\theta_1\wedge\theta_2  =  \frac{1}{2\pi} \int_0^{2\pi}\int_0^{2\pi}k(t) dsdt 
%
%
%
=   \int_0^{2\pi} k(t )dt.    
\end{eqnarray*}
%
%
%
\textbf{Step 4.} For any $f,g\in C^\infty(\T^2)$ and   $\gamma= f \theta_1 + g\theta_2 \in\Omega^1(\T^2)$ we have 
\begin{eqnarray*}
\int_{\T^2}\alpha\wedge \gamma =   \frac{1}{2\pi}\int_{\T^2} f \theta_1\wedge \theta_2 + \frac{1}{2\pi}\int_{\T^2}g\theta_2\wedge \theta_2   =  \frac{1}{2\pi}\int_{\T^2} f \theta_1\wedge \theta_2. 
\end{eqnarray*}
On the other hand since $i^*\theta_2=0$ we get
\[
\int_N\gamma  = \int_Nf\theta_1 + \int_N g\theta_2 = \int_N f\theta_1 .
\]
Then, it suffices to  prove that $\int_N f \theta_1=\frac{1}{2\pi}\int_{\T^2} f \theta_1\wedge \theta_2$.
The condition $d\gamma=0$ implies  that 
\[
df\wedge \theta_1 + dg\wedge\theta_2=(-\partial_2f+\partial_1 g)\theta_1\wedge\theta_2=0
\]
i.e. $\partial_2f=\partial_1g$. Set $h(s):=\int_0^{2\pi}\Phi^*f(t,s)dt$. We show that $h(s)$ is a constant function in $s$. Since $g $ is a periodic function with period $2\pi$,  we observe that
\begin{eqnarray*}
\partial_sh(s)  &=&  \int_0^{2\pi}\partial_s\Phi^*f(t,s)dt =\int_0^{2\pi}\Phi^*\partial_2f(t,s)dt\\
&=&  \int_0^{2\pi} \Phi^*\partial_1g(t,s)dt \\
&=&  \int_0^{2\pi}\partial_t\Phi^*g(t,s)dt \\
&=&  \Phi^*g(2\pi,s) -\Phi^*g(0,s) =0
\end{eqnarray*}
As a result $h(s)=h(s_0)$  or equivalently $\int_0^{2\pi}\Phi^*f(t,s)dt=\int_0^{2\pi}\Phi^*f(t,s_0)dt$. However,
\begin{eqnarray*}
\frac{1}{2\pi}\int_{\T^2} f \theta_1\wedge \theta_2  &=&  \frac{1}{2\pi}\int_0^{2\pi}\int_0^{2\pi} \Phi^*f(t,s)dt ds\\
&=&  \frac{1}{2\pi}\int_0^{2\pi}\int_0^{2\pi} \Phi^*f(t,s_0)ds dt\\
&=& \int_0^{2\pi} \Phi^*f(t,s_0)dt\\
&=& \int_N f\theta_1
\end{eqnarray*}
which completes the proof.
\end{proof}\noindent
%
%
%
%
\textbf{Proof of proposition \ref{prop contractio form}.}\\
\begin{proof}
Using the expression for Lie bracket on the central extension and also equation \eqref{covariant derivative}, we have
\begin{equation*}
\wh{\nabla}_{[(u,a),(X,0)]_{\hat{\cg}}}(X,0)=\Big( \nabla_{[u,X]}X -\frac{1}{2}\omega(u,X)TX ~,~ \frac{1}{2}\omega([u,X],X)  \Big).
\end{equation*}
Moreover using the fact that $\omega(X,u)=-\omega(u,X)$ and equation \eqref{second covariant derivative} we have
\begin{eqnarray*}
&& \wh{R}\Big(  (u,a) , (X,0)  \Big) (X,0) \\
& =&  \wh{\nabla}_{(u,a)}   \wh{\nabla}_{(X,0)} (X,0)  -     \wh{\nabla}_{(X,0)} \wh{\nabla}_{(u,a)} (X,0)   -  \wh{\nabla}_{[(u,a),(X,0)]_{\hat{\cg}}}(X,0)\\
& =&  \Big( \nabla_u   \nabla_XX  -\frac{1}{2}aT\nabla_XX  ~,~ \frac{1}{2}\omega(u,\nabla_XX)   \Big)  \\
&& -\Big( \nabla_X \nabla_uX -\frac{1}{2}\nabla_X(aTX) - \frac{1}{4}\omega(u,X)TX~,~  \frac{1}{2}\omega(X,\nabla_uX -\frac{1}{2}aTX )\Big)\\
&&  -  \Big(   \nabla_{[u,X]}X -\frac{1}{2}\omega(u,X)TX ~,~ \frac{1}{2}\omega([u,X],X) \Big)\\
\end{eqnarray*}
\begin{eqnarray*}
& =&  \Big( R(u,X)X  -\frac{1}{2}aT\nabla_XX     +        \frac{1}{2}\nabla_X(aTX)  + \frac{1}{4}\omega(u,X)TX 
+     \frac{1}{2}\omega(u,X)TX   \\
&& ~,~ \frac{1}{2}\omega(u,\nabla_XX)     -   \frac{1}{2}\omega(X,\nabla_uX -\frac{1}{2}aTX )    -   \frac{1}{2}\omega([u,X],X)   \Big)  \\
& =&  \Big( R(u,X)X  -\frac{1}{2}aT\nabla_XX     +        \frac{1}{2}\nabla_X(aTX)  + \frac{3}{4}\omega(u,X)TX    \\
&& ~,~ \frac{1}{2}\omega(u,\nabla_XX)     -   \frac{1}{2}\omega(X,\nabla_uX -[u,X] -\frac{1}{2}aTX )       \Big)  
\end{eqnarray*}
On the other hand
\begin{eqnarray*}
\wh{\nabla}_{(X,0)}\wh{\nabla}_{(X,0)}(u,a) = \Big(  \nabla_X\nabla_Xu   - \frac{1}{2}\nabla_X(aTX)  - \frac{1}{4}\omega(X,u)TX  ~,~   
   \frac{1}{2}\omega(X,\nabla_Xu -\frac{1}{2}aTX )\Big).
\end{eqnarray*}
As a result we get
\begin{eqnarray*}
&& 2 \widehat{K}\big( (u,a),(X,0) \big) \\
&=&    \Big(  \nabla_X\nabla_Xu   - \frac{1}{2}\nabla_X(aTX)  - \frac{1}{4}\omega(X,u)TX  ~,~   
   \frac{1}{2}\omega(X,\nabla_Xu -\frac{1}{2}aTX )\Big)\\
&&  +  \Big( R(u,X)X  -\frac{1}{2}aT\nabla_XX     +        \frac{1}{2}\nabla_X(aTX)  + \frac{3}{4}\omega(u,X)TX    \\
&& ~,~ \frac{1}{2}\omega(u,\nabla_XX)     -   \frac{1}{2}\omega(X,\nabla_uX -[u,X] -\frac{1}{2}aTX )       \Big)  \\
&=& \Big( \nabla_X\nabla_Xu  + R(u,X)X  -\frac{1}{2}aT\nabla_XX    +   \omega(u,X)TX    ~,~    \frac{1}{2}\omega(u,\nabla_XX)   \Big) \\            
&=&\Big( 2K(u,X)  -\frac{1}{2}aT\nabla_XX    +   \omega(u,X)TX    ~,~    \frac{1}{2}\omega(u,\nabla_XX)   \Big) 
\end{eqnarray*}
\end{proof}\noindent
%
%
%
%
\textbf{Proof of lemma \ref{lemma sigma}.}\\
\begin{proof}
\textbf{i}. The orthogonal set $\{A_k,B_k\}_k$  form a basis for the space of Hamiltonian vector fields. On the other hand
\begin{eqnarray*}
\ll TA_k, A_{k^\prime} \gg_\cg &=& \omega(A_k,A_{k^\prime}) \\
&=&  \int_{\T^2} - \lambda(|k|)\sin(k.\theta) \alpha(A_{k^\prime}) d\theta\\
&=&  \beta \lambda(|k|)^2 k_1   \int_{\T^2} \sin(k.\theta) \cos(k^\prime.\theta) d\theta   = 0 
\end{eqnarray*}
Similarly we have,
\begin{eqnarray*}
\ll TA_k, B_{k^\prime} \gg_\cg &=& \omega(A_k,B_{k^\prime}) \\
&=&  \int_{\T^2} - \lambda(|k|)\sin(k.\theta) \alpha(B_{k^\prime}) d\theta\\
&=&  \beta \lambda(|k|)^2 k_1   \int_{\T^2} \sin(k.\theta) \sin(k^\prime.\theta) d\theta   =   \beta k_1\lambda(|k|)^2   N \delta_{k^\prime}^k
\end{eqnarray*}
where $\delta_{k^\prime}^k$ is $1$ when $k=k^\prime$ and is zero otherwise. As a result $TA_k =  \beta k_1\lambda(|k|)^2N B_k$.\\
\textbf{ii.}The proof is similar to that of part \textbf{i}.\\
\textbf{iii. } \&  \textbf{iv.}
\begin{eqnarray*}
\omega(u,A_k)  &=&  \int_{\T^2} \psi_u \alpha(A_{k}) d\theta\\
&=&  -\beta \lambda(|k|) k_1   \int_{\T^2} \psi_u \cos(k.\theta) d\theta    =  - \beta \lambda(|k|) k_1 (\psi_u,\cos k.\theta)_{L^2},
\end{eqnarray*}
\begin{eqnarray*}
\omega(u,B_k)  &=&  \int_{\T^2} \psi_u \alpha(B_{k}) d\theta\\
&=&  -\beta \lambda(|k|) k_1   \int_{\T^2} \psi_u \sin(k.\theta) d\theta    =  - \beta \lambda(|k|) k_1 (\psi_u,\sin k.\theta)_{L^2}.
\end{eqnarray*}
\textbf{v.} Note that for any $k\in\mathbb{Z}^2$ 
\begin{eqnarray*}
\ll u, A_{k} \gg_\cg   &=&  \ll \nabla^\perp \psi_u  , \nabla^\perp (-\lambda_k\sin k.\theta  )\gg_\cg \\
&=& -\lambda(|k|) (\psi_u , -\Delta\sin k.\theta)_{L^2}\\
&=&   -\lambda(|k|) (k_1^2+k_2^2)(\psi_u , \sin k.\theta)_{L^2}
\end{eqnarray*}
and similarly 
\[
\ll u, B_{k}  \gg_\cg  =\lambda(|k|) (k_1^2+k_2^2)(\psi_u , \cos k.\theta)_{L^2}.
\]
As a result we have
\begin{eqnarray*}
\omega(u,A_k) TA_k &=&   - \beta^2 \lambda(|k|)^3 k_1^2 A (\psi_u,\cos k.\theta)_{L^2} B_k\\
&=&   -\frac{ \beta^2 \lambda(|k|)^3 k_1^2 A }{   \lambda(|k|) (k_1^2+k_2^2)}    \ll u, B_{k}  \gg_\cg     B_k\\
&=&   -\frac{ \beta^2 \lambda(|k|)^2 k_1^2  A}{    (k_1^2+k_2^2)}    \ll u, B_{k}  \gg_\cg     B_k 
\end{eqnarray*}
and similarly
\[
\omega(u,B_k) TB_k  = -\frac{ \beta^2 \lambda(|k|)^2 k_1^2 A }{    (k_1^2+k_2^2)}    \ll u, A_{k}  \gg_\cg     A_k 
\]
holds true. 

Set $\bar{B}_k:=\frac{  \lambda(|k|) k_1  }{   \sqrt{ k_1^2+k_2^2}} B_k$ and  $\bar{A}_k:=\frac{  \lambda(|k|) k_1  }{   \sqrt{ k_1^2+k_2^2}} A_k$. Since  $\{\bar{A},\bar{B}\}_{k\in\mathbb{Z}^2}$ form a basis for $\cg$ we get
\begin{eqnarray*}
&&  \sum_k\omega(u,A_k)TA_k+\omega(u,B_k)TB_k  =  \omega(u,A_k) TA_k\\
&=&   \sum_k   -\frac{ \beta^2 \lambda(|k|)^2 k_1^2  N}{    (k_1^2+k_2^2)}    \ll u, B_{k}  \gg_\cg     B_k    
-\frac{ \beta^2 \lambda(|k|)^2 k_1^2 A }{    (k_1^2+k_2^2)}    \ll u, A_{k}  \gg_\cg     A_k \\
&=&  - \beta^2   N  \sum_k   \frac{  \lambda(|k|)^2 k_1^2  }{    (k_1^2+k_2^2)}    \ll u, B_{k}  \gg_\cg     B_k    
+\frac{  \lambda(|k|)^2 k_1^2   }{    (k_1^2+k_2^2)}    \ll u, A_{k}  \gg_\cg     A_k \\
&=&  - \beta^2   N  \sum_k       \ll u, \Bar{B}_{k}  \gg_\cg     \Bar{B}_k  +     \ll u, \Bar{A}_{k}  \gg_\cg     \Bar{A}_k \\
&=&  -\beta^2N u :=-2\sigma u
\end{eqnarray*}
\end{proof}
\noindent\textbf{Acknowledgements}. I would like to express my gratitude to Ana Bela Cruzeiro, Cornelia Vizman and Bas Janssens  for their invaluable comments and suggestions, which have significantly improved the quality of this work.\\
Funded by the Deutsche Forschungsgemeinschaft (DFG, German
Research Foundation), project  517512794.
%
%
%
%
%

%
Ali Suri, Universit\"{a}t Paderborn, Warburger Str. 100,
33098 Paderborn, Germany.\\
email: asuri@math.upb.de\vfill
\end{document}